\documentclass[final,leqno,onefignum,onetabnum]{siamltex}

\usepackage{psfrag}
\usepackage{graphicx}
\usepackage{amssymb}
\usepackage{amsmath}

\makeatletter
\def\boldmatha{\mathversion{bold}}
\def\boldmath#1{\mathchoice
          {\mbox{\boldmatha$\displaystyle#1$}}%
         {\mbox{\boldmatha$#1$}}%
          {\mbox{\boldmatha$\scriptstyle#1$}}%
          {\mbox{\boldmatha$\scriptscriptstyle#1$}}}
\makeatother

\newcommand{\cgq}{$\mathrm{cG}(q)$}
\newcommand{\dgq}{$\mathrm{dG}(q)$}
\newcommand{\mcgq}{$\mathrm{mcG}(q)$}
\newcommand{\mdgq}{$\mathrm{mdG}(q)$}
\newcommand{\mcgqd}{$\mathrm{mcG}(q)^*$}
\newcommand{\mdgqd}{$\mathrm{mdG}(q)^*$}
\newcommand{\picg}{\pi_{\mathrm{cG}}}
\newcommand{\pidg}{\pi_{\mathrm{dG}}}
\newcommand{\picgq}{\pi_{\mathrm{cG}}^{[q]}}
\newcommand{\pidgq}{\pi_{\mathrm{dG}}^{[q]}}
\newcommand{\picgqd}{\pi_{\mathrm{cG^*}}^{[q]}}
\newcommand{\pidgqd}{\pi_{\mathrm{dG^*}}^{[q]}}

\newcommand{\pidgd}{\pi_{\mathrm{dG^*}}}
\newcommand{\real}{\mathbb{R}}

\newcommand{\naturals}{\mathbb{N}}
\newcommand{\sgn}{\mathrm{sgn}}

\newcommand{\Od}[3]{\frac{d^{#1}#2}{d#3^{#1}}}

\newcommand{\PD}[2]{\frac{\partial#1}{\partial#2}}

\newtheorem{remark}{Remark}[section]

\renewcommand{\ldots}{\dotsc}
\begin{document}

\title{Multiadaptive Galerkin Methods for ODEs III:\\
A~Priori Error Estimates\thanks{Received
by the editors February 12, 2004; accepted for publication (in revised form) May 4, 2005;
published electronically January 27, 2006.
\URL sinum/43-6/60413.html}}

 \author{Anders Logg\thanks{Toyota Technological Institute at Chicago,
        1427 East 60th Street, Chicago, IL 60637 (logg@\break tti-c.org).}}
\date{\today}

\slugger{sinum}{2006}{43}{6}{2624--2646}
\maketitle

\setcounter{page}{2624}

\begin{abstract}
  The multiadaptive continuous/discontinuous Galerkin methods \mcgq{}
  and \mdgq{} for the numerical solution of initial value problems for
  ordinary differential equations are based on piecewise polynomial
  approximation of degree $q$ on partitions in time with time steps
  which may vary for different components of the computed solution.
  In this paper, we prove general order a priori error estimates for
  the \mcgq{} and \mdgq{} methods. To prove the error estimates, we
  represent the error in terms of a discrete dual solution and the
  residual of an interpolant of the exact solution. The estimates then
  follow from interpolation estimates, together with stability
  estimates for the discrete dual solution.
\end{abstract}

\begin{keywords}
  multiadaptivity, individual time steps, local time steps,
  ODE, continuous Galerkin, discontinuous Galerkin, mcG($q$), mdG($q$), a priori
  error estimates, existence, stability, Peano kernel theorem,
  interpolation estimates, piecewise smooth
\end{keywords}

\begin{AMS}
65L05, 65L07,  65L20,  65L50,  65L60,  65L70
\end{AMS}

\begin{DOI}
10.1137/040604133
\end{DOI}

\pagestyle{myheadings} \thispagestyle{plain}
 \markboth{ANDERS LOGG}{MULTIADAPTIVE GALERKIN METHODS FOR ODES III}

\section{Introduction}
\label{sec:intro}

This is part 3 in a sequence of papers
\cite{logg:article:01,logg:article:02} on
multiadaptive Galerkin methods, \mcgq\ and \mdgq,
for approximate (numerical) solution of ODEs of the form
\begin{equation}
    \begin{split}
      \dot{u}(t) &= f(u(t),t), \quad t\in(0,T], \\
      u(0) &= u_0,
    \end{split}
  \label{eq:u'=f}
\end{equation}
where $u : [0,T] \rightarrow \real^N$ is the solution to be computed,
$u_0 \in \real^N$ a given initial condition,
$T>0$ a given final time,
and $f : \real^N \times (0,T] \rightarrow \real^N$ a given
function that is Lipschitz-continuous in $u$ and bounded.

In the previous two parts of our series on multiadaptive Galerkin methods, we
proved a posteriori error estimates, through which the time steps are
adaptively determined from residual feedback and stability
information, obtained by solving a dual linearized problem.
In this paper, we prove a priori error estimates for \mcgq{} and
\mdgq{}. We also prove the stability estimates and interpolation
estimates which are essential to the a priori error analysis.

Standard methods for the time-discretization of (\ref{eq:u'=f})
require that the resolution is equal for all components $U_i(t)$ of
the computed approximate solution $U(t)$ of (\ref{eq:u'=f}). This
includes all standard Galerkin or Runge--Kutta methods; see
\cite{EriEst96,But87,HaiWan91a,HaiWan91b,Sha94,Pet98}. Using the same
time step sequence $k = k(t)$ for all components could become
very costly if the different components of the solution exhibit
multiple time scales of different magnitudes. We therefore propose
a new representation
of the solution in which the difference in time scales is
reflected in the \emph{componentwise} time-discretization of
(\ref{eq:u'=f}), that is, each component $U_i(t)$ is computed using an
individual time step sequence $k_i = k_i(t)$.

The multiadaptive Galerkin methods \mcgq{} and \mdgq{} first
presented in \cite{logg:article:01} are formulated as extensions of
the standard continuous and discontinuous Galerkin methods \cgq{} and
\dgq{}, studied earlier in detail by Hulme \cite{Hul72a,Hul72b}, Jamet
\cite{Jam78}, Delfour, Hager, and Trochu \cite{DelHag81}, Eriksson,
Johnson, and Thom\'ee \cite{EriJoh85,Joh88,EriJoh91,EriJoh95a,EriJohIII,EriJoh95b,EriJoh95c,EriJoh98,EriEst95},
and Estep et al.\
\cite{Est95,EstFre94,EstLar00,EstWil96,EstStu02}.
As such, the analysis of the \mcgq{} and \mdgq{} methods can be carried out within
the existing framework, but the extension to multiadaptive
time-stepping leads to some technical challenges, in particular,  proving
the appropriate interpolation estimates.

Local (multiadaptive) time-stepping has been explored before to some
extent for specific applications, including specialized integrators
for the $n$-body problem \cite{Mak92,DavDub97,AleAgn98} and low-order
methods for conservation laws \cite{OshSan83,FlaLoy97,DawKir01}.
Early attempts at local time-stepping include
\cite{HugLev83a,HugLev83b}. Recently, a new class of related methods,
known as asynchronous variational integrators (AVI) with local
time steps, has been proposed \cite{LewMar03}.

\subsection{Main results}

The main results of this paper are a priori error estimates for the
\mcgq\ and \mdgq\ methods, respectively, of the form
\begin{equation}
  \|e(T)\|_{l_p} \leq C S(T) \big\|k^{2q} u^{(2q)}\big\|_{L_{\infty}([0,T],l_1)}
\end{equation}
and
\begin{equation}
  \|e(T)\|_{l_p} \leq C S(T) \big\|k^{2q+1} u^{(2q+1)}\big\|_{L_{\infty}([0,T],l_1)}
\end{equation}
for $p=2$ or $p=\infty$, where $C$ is an interpolation constant, $S(T)$ is a (computable) stability factor, and
$k^{2q}u^{(2q)}$ (or $k^{2q+1}u^{(2q+1)}$) combines local time steps $k_i = k_i(t)$ with derivatives
of the exact solution $u$. The norm $L_{\infty}(I,\|\cdot\|)$ is defined by
$\|v\|_{L_{\infty}(I,\|\cdot\|)} = \sup_{t \in I} \|v(t)\|$.
These estimates state that the \mcgq{} method is of order $2q$ and
that the \mdgq{} method is of order $2q+1$ in the local time step.
We refer to section \ref{sec:aprioriestimates} for
the exact results. It should be noted that superconvergence is obtained only
at synchronized time levels, such as the end-point $t = T$.
For the general nonlinear problem, we obtain exponential
estimates for the stability factor $S(T)$. In \cite{logg:thesis:03}, we prove
that for a parabolic model problem, the stability
factor remains bounded and of unit size, independent of $T$ (up to a logarithmic factor).

\subsection{Notation}

The following notation is used throughout this paper.
Each component $U_i(t)$, $i=1,\ldots,N$, of the approximate
$\mathrm{m(c/d)G}(q)$ solution $U(t)$ of (\ref{eq:u'=f}) is a
piecewise polynomial on a partition of $(0,T]$ into $M_i$ subintervals.
Subinterval $j$ for component $i$ is denoted by $I_{ij}=(t_{i,j-1},t_{ij}]$,
and the length of the subinterval is given by the \emph{local time step} $k_{ij}=t_{ij}-t_{i,j-1}$.
This is illustrated in Figure \ref{fig:intervals}.
On each subinterval $I_{ij}$, $U_{i}\vert_{I_{ij}}$ is a polynomial
of degree $q_{ij}$ and we refer to $(I_{ij},U_i\vert_{I_{ij}})$ as an \emph{element}.

Furthermore, we shall assume that the interval $(0,T]$ is
partitioned into blocks between certain synchronized time levels
$0=T_0<T_1<\cdots<T_M=T$. We refer to the set of intervals $\mathcal{T}_n$ between
two synchronized time levels $T_{n-1}$ and $T_n$ as a \emph{time
slab}:
\begin{displaymath}
	\mathcal{T}_n = \{ I_{ij} : T_{n-1} \leq t_{i,j-1} < t_{ij} \leq T_n \}.
\end{displaymath}
We denote the length of a time slab by $K_n = T_n - T_{n-1}$.
We also refer to the entire collection of intervals $I_{ij}$ as the partition $\mathcal{T}$.

Since different components use different time steps, a local interval $I_{ij}$ may
contain nodal points for other components, that is, some $t_{i'j'}\in (t_{i,j-1},t_{ij})$.
We denote the set of such internal nodes on a local interval $I_{ij}$ by $\mathcal{N}_{ij}$.

\begin{figure}[t]
	\begin{center}
		\psfrag{0}{$0$}
		\psfrag{i}{$i$}
		\psfrag{k}{$k_{ij}$}
		\psfrag{K}{$K_n$}
		\psfrag{T}{$T$}
		\psfrag{I}{$I_{ij}$}
		\psfrag{t1}{$t_{i,j-1}$}
		\psfrag{t2}{$t_{ij}$}
		\psfrag{T1}{$T_{n-1}$}
		\psfrag{T2}{$T_n$}
                \psfrag{t}{$t$}
		\includegraphics[width=10cm]{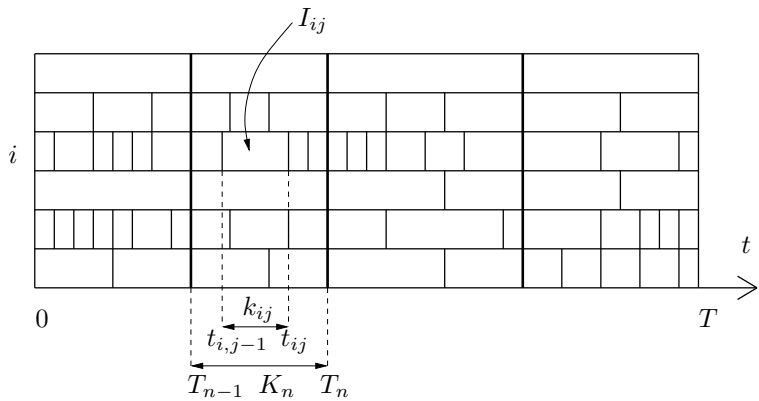}
		\caption{Individual partitions of the interval $(0,T]$ for different components. Elements
					between common synchronized time levels are organized
					in time slabs. In this example, we have $N=6$ and $M=4$.}
		\label{fig:intervals}
	\end{center}
\end{figure}

\subsection{Outline of the paper}

The outline of this paper is as follows. In section \ref{sec:definition},
we give the full definition of the multiadaptive Galerkin methods
\mcgq{} and \mdgq{}. We also introduce the dual methods \mcgqd{} and
\mdgqd{},  which are of importance to the a priori error analysis. In
sections \ref{sec:existence} and \ref{sec:stability}, respectively,
we then prove existence and stability of the discrete
solutions as defined in section \ref{sec:definition}.

In section \ref{sec:interpolation}, we prove the interpolation
estimates that we later use to prove the a~priori error estimates in
section~\ref{sec:apriori}. Proving the interpolation estimates is
technically challenging, since the function to be interpolated may be
discontinuous within the interval of interpolation. To measure the
regularity of the interpolated function, it is then necessary to take
into consideration the size of the jump in function value and
derivatives at each point of discontinuity.

Finally, in section \ref{sec:numerical}, we
present some numerical evidence for the a priori error estimates by
solving a simple model problem and showing that we obtain the
predicted convergence rates, $k^{2q}$ and $k^{2q+1},$ respectively, for
the \mcgq{} and \mdgq{} methods.

\section{Definition of methods}
\label{sec:definition}

In this section, we give the definitions of the multiadaptive
Galerkin methods \mcgq{} and \mdgq{}. The multiadaptive methods are
obtained as extensions of the standard (monoadaptive) Galerkin
methods \cgq{} and \dgq{} by extending the trial and test spaces to
allow individual time step sequences for different components.

As an important tool for the a priori error analysis in section~\ref{sec:apriori},
we also introduce the discrete dual problem and the
discrete dual methods \mcgqd{} and \mdgqd{}.

\subsection{Multiadaptive continuous Galerkin, mcG($\boldmath q$)}

To formulate the \mcgq\ method, we define the
\emph{trial space} $V$ and the \emph{test space} $\hat{V}$ as
\begin{equation}
  \begin{split}
    V &= \big\{v \in [\mathcal{C}([0,T])]^N : v_i|_{I_{ij}}\in \mathcal{P}^{q_{ij}}(I_{ij}), \
    j=1,\ldots,M_i, \ i=1,\ldots,N \big\},\\
    \hat{V} &= \big\{v : v_i|_{I_{ij}}\in \mathcal{P}^{q_{ij}-1}(I_{ij}), \
    j=1,\ldots,M_i, \ i=1,\ldots,N \big\},\\
  \end{split}
  \label{eq:spaces,mcg}
\end{equation}
where $\mathcal{P}^q(I)$ denotes the linear space of polynomials of
degree $q$ on an interval $I\subset\real$. In other words, $V$ is the space
of vector-valued continuous piecewise polynomials of degree $q = (q_i(t))$
with $q_i(t) \geq 1$ on the partition $\mathcal{T}$,
and $\hat{V}$ is the space of vector-valued (possibly discontinuous) piecewise polynomials
of degree $q-1 = (q_i(t) - 1)$ on the same partition.

We now define the \mcgq\ method for (\ref{eq:u'=f}) as follows:
Find $U\in V$ with $U(0)=u_0$ such that
\begin{equation}
         \int_0^T (\dot{U},v)\, dt = \int_0^T (f(U,\cdot),v) \, dt \quad
    \forall v\in \hat{V},
    \label{eq:fem,mcg}
\end{equation}
where $(\cdot,\cdot)$ denotes the $\real^N$ inner product.  With a
suitable choice of test function $v$, it follows that the global
problem (\ref{eq:fem,mcg}) can be restated as a sequence of successive
local problems for each component: For $i=1,\ldots,N$,
$j=1,\ldots,M_i$, find $U_i|_{I_{ij}}\in \mathcal{P}^{q_{ij}}(I_{ij})$
with $U_i(t_{i,j-1})$ given such that
\begin{equation}
    \int_{I_{ij}} \dot{U}_i v \, dt = \int_{I_{ij}} f_i(U,\cdot) v \, dt \quad
    \forall v\in \mathcal{P}^{q_{ij}-1}(I_{ij}),
    \label{eq:fem,mcg,local}
\end{equation}
where the initial condition is specified for $i=1,\ldots,N$ by $U_i(0)=u_i(0)$.

We define the \emph{residual} $R$ of the approximate solution $U$
by $R_i(U,t) = \dot{U}_i(t) - f_i(U(t),t)$. In terms of the residual,
we can rewrite (\ref{eq:fem,mcg,local}) in the form
\begin{equation}
    \int_{I_{ij}} R_i(U,\cdot) v \, dt = 0 \quad
    \forall v\in \mathcal{P}^{q_{ij}-1}(I_{ij}),
    \quad j = 1,\ldots,M_i, \quad i = 1,\ldots,N,
    \label{eq:fem,mcg,local,orthogonality}
\end{equation}
that is, the residual is orthogonal to the test space on each local interval. We refer
to (\ref{eq:fem,mcg,local,orthogonality}) as the \emph{Galerkin orthogonality} of the \mcgq\ method.

\subsection{Multiadaptive discontinuous Galerkin, mdG($\boldmath q$)}

For \mdgq{}, we define the trial and test spaces by
\begin{equation}
    V = \hat{V} = \big\{v : v_i|_{I_{ij}}\in \mathcal{P}^{q_{ij}}(I_{ij}), \
    j=1,\ldots,M_i, \ i=1,\ldots,N \big\},
  \label{eq:spaces,mdg}
\end{equation}
that is, both trial and test functions are
vector-valued (possibly discontinuous) piecewise polynomials
of degree $q = (q_i(t))$ with $q_i(t) \geq 0$ on the partition $\mathcal{T}$.
By definition, the \mdgq\ solution $U\in V$ is left-continuous.

We now define the \mdgq\ method for (\ref{eq:u'=f}) as follows:
Find $U\in V$ with $U(0^-)=u_0$ such that
\begin{equation}
  \sum_{i=1}^N \sum_{j=1}^{M_i}
  \left[
    [U_i]_{i,j-1} v_i\big(t_{i,j-1}^+\big) + \int_{I_{ij}} \dot{U}_i v_i \, dt
  \right] =
  \int_0^T (f(U,\cdot),v) \, dt \quad \forall v\in \hat{V},
  \label{eq:fem,mdg}
\end{equation}
where \smash{$[U_i]_{i,j-1} = U_i(t_{i,j-1}^+) - U_i(t_{i,j-1}^-)$} denotes the jump in $U_i(t)$ across
the node \smash{$t=t_{i,j-1}$}, and where \smash{$v(t^+) = \lim_{s \rightarrow t^+} v(s)$}.

The \mdgq\ method in local form, corresponding to
(\ref{eq:fem,mcg,local}), reads as follows:
For $i=1,\ldots,N$, $j=1,\ldots,M_i$, find $U_i|_{I_{ij}}\in \mathcal{P}^{q_{ij}}(I_{ij})$ such that
\begin{equation}
    [U_i]_{i,j-1} v(t_{i,j-1}) + \int_{I_{ij}} \dot{U}_i v \, dt = \int_{I_{ij}} f_i(U,\cdot) v \, dt
	\quad  \forall v\in \mathcal{P}^{q_{ij}}(I_{ij}),
    \label{eq:fem,mdg,local}
\end{equation}
where the initial condition is specified
for $i=1,\ldots,N$ by $U_i(0^-) = u_i(0)$.

The residual $R$ is defined on the inner of each local interval
$I_{ij}$ by $R_i(U,t)=\dot{U}_i(t) - f_i(U(t),t)$. In terms of the
residual, (\ref{eq:fem,mdg,local}) can be restated in the form
\begin{equation}
    [U_i]_{i,j-1} v\big(t_{i,j-1}^+\big) + \int_{I_{ij}} R_i(U,\cdot) v \, dt = 0 \quad
    \forall v\in \mathcal{P}^{q_{ij}}(I_{ij})
    \label{eq:fem,mdg,local,orthogonality}
\end{equation}
for $j = 1,\ldots,M_i$, $i = 1,\ldots,N$.  We refer to
(\ref{eq:fem,mdg,local,orthogonality}) as the Galerkin orthogonality
of the \mdgq\ method.

\subsection{The dual problem}

The dual problem is the standard tool for error analysis, a priori or
a posteriori, of Galerkin finite element methods for the numerical
solution of differential equations; see \cite{EriEst95,BecRan01}.
For the a posteriori error analysis of the multiadaptive Galerkin methods
\mcgq{} and \mdgq{} in \cite{logg:article:01}, we formulate a
continuous dual problem. For the a priori error analysis of this
paper, we formulate instead a discrete dual problem.
The discrete dual problem was first introduced for the family of discontinuous
Galerkin methods \dgq{} in \cite{EriJoh85}.
As we shall see, the discrete dual problem can be expressed as a Galerkin
method for a continuous problem.

The discrete dual solution $\Phi:[0,T]\rightarrow \real^N$ is a Galerkin approximation of
the exact solution $\phi:[0,T]\rightarrow \real^N$ of the continuous
dual problem
\begin{equation}
  \begin{split}
    -\dot{\phi}(t) &= J^{\top}(\pi u,U,t)\phi(t) + g(t), \quad t \in [0,T), \\
      \phi(T) &= \psi,
  \end{split}
  \label{eq:dual}
\end{equation}
where $\pi u$ is an interpolant or a projection of the exact solution
$u$ of (\ref{eq:u'=f}), $g:[0,T]\rightarrow \real^N$ is a given
function, $\psi\in\real^N$ is a given initial condition, and
\begin{equation}
  J^{\top}(\pi u,U,t) =
  \left(
    \int_0^1 \PD{f}{u}(s \pi u(t)+(1-s) U(t),t) \, ds\right){\!\!^{\top}}\!,
  \label{eq:J}
\end{equation}
that is, an appropriate mean value of the transpose of the Jacobian of
the right-hand side $f(\cdot,t)$ evaluated at $\pi u(t)$ and
$U(t)$. Note that by the chain rule, we have
\begin{equation}
	J(\pi u,U,\cdot) (U - \pi u) = f(U,\cdot) - f(\pi u,\cdot).
	\label{eq:Ju}
\end{equation}
The data $(\psi,g)$ of the dual problem allow us to obtain error
estimates for different functionals $L_{\psi,g}$ of the
error $e = U - u$.

We define below two new Galerkin methods for the dual problem (\ref{eq:dual}):
the dual methods \mcgqd{} and \mdgqd{}. We will later use the \mcgqd{}
method to express the error of the \mcgq{} solution of
(\ref{eq:u'=f}) in terms of the \mcgqd{} solution of
(\ref{eq:dual}). Similarly, we will express the error of the \mdgq{}
solution of (\ref{eq:u'=f}) in terms of the \mdgqd{} solution of (\ref{eq:dual}).

\subsection{Multiadaptive dual continuous Galerkin, { {mcG($\boldmath q)^*$}}}

In the formulation of the dual method of \mcgq, we interchange the trial and test spaces of \mcgq{}.
With the same definitions of $V$ and $\hat{V}$ as in (\ref{eq:spaces,mcg}),
we thus define the \mcgqd\ method for (\ref{eq:dual}) as follows:
Find $\Phi\in \hat{V}$ with $\Phi(T^+)=\psi$ such that
\begin{equation} \label{eq:mcgqd}
  \int_0^T (\dot{v},\Phi) \, dt =
  \int_0^T (J(\pi u,U,\cdot) v, \Phi) + L_{\psi,g}(v)
\end{equation}
for all $v\in V$ with $v(0)=0$, where
\begin{equation}
  L_{\psi,g}(v) \equiv (v(T),\psi) + \int_0^T (v,g) \, dt.
\end{equation}
Notice the extra condition that the test functions should vanish at $t=0$, which is introduced
to make the dimension of the test space equal to the dimension of the trial space.
Integrating by parts, (\ref{eq:mcgqd}) can alternatively be expressed in the form
\begin{equation} \label{eq:mcgqd,alternative}
  \sum_{i=1}^N \sum_{j=1}^{M_i} \left[ - [\Phi_i]_{ij} v_i(t_{ij}) - \int_{I_{ij}} \dot{\Phi}_i v_i \, dt \right] =
  \int_0^T (J^{\top}(\pi u,U,\cdot) \Phi + g,v) \, dt.
\end{equation}

\subsection{Multiadaptive dual discontinuous Galerkin,{ mdG($\boldmath q)^*$}}

$\!\!$With the same definitions of $V$ and $\hat{V}$ as in (\ref{eq:spaces,mdg}),
we define the \mdgqd\ method for (\ref{eq:dual}) as follows:
Find $\Phi\in \hat{V}$ with $\Phi(T^+)=\psi$ such that
\begin{equation} \label{eq:mdgqd}
  \sum_{i=1}^N \sum_{j=1}^{M_i}
  \left[
    [v_i]_{i,j-1} \Phi_i\big(t_{i,j-1}^+\big) + \int_{I_{ij}} \dot{v}_i \Phi_i \, dt
  \right] =
  \int_0^T (J(\pi u,U,\cdot)v,\Phi) \, dt +  L_{\psi,g}(v)
\end{equation}
for all $v\!\in\! V$ with $v(0^-) \!= \!0$.
Integrating by parts, (\ref{eq:mdgqd}) can alternatively
be expressed in the form
\begin{equation} \label{eq:mdgqd,alternative}
  \sum_{i=1}^N \sum_{j=1}^{M_i}
  \left[
    - [\Phi_i]_{ij} v_i\big(t_{ij}^-\big) - \int_{I_{ij}} \dot{\Phi}_i v_i \, dt
  \right] =
  \int_0^T (J^{\top}(\pi u,U,\cdot)\Phi+g,v) \, dt.
\end{equation}

\section{Existence of solutions}
\label{sec:existence}

To prove existence of the discrete \mcgq, \mdgq, \mcgqd, and \mdgqd\
solutions defined in the previous section, we formulate fixed point
iterations for the construction of solutions. Existence then follows
from the Banach fixed point theorem if the time steps are
sufficiently small.

\begin{lemma}[fixed point iteration]
  \label{lem:explicit}
  Let $\mathcal{T}_n$ be a time slab with synchronized time levels
  $T_{n-1}$ and $T_n$.
  With
  time reversed for the dual methods (to simplify the notation),
  the \mcgq, \mdgq, \mcgqd, and \mdgqd\ methods can all be expressed
  in the following form:
  For all $I_{ij}\in \mathcal{T}_n$, find
  $\{\xi_{ijn}\}$ (the degrees of freedom for $U_i$ on $I_{ij}$) such that
  \begin{equation}
    \label{eq:explicit}
    \xi_{ijn} = u_i(0) +
    \int_{0}^{t_{i,j-1}} f_i(U,\cdot) \, dt +
    \int_{I_{ij}} w_n^{[q_{ij}]}(\tau_{ij}(t)) f_i(U,\cdot) \, dt,
  \end{equation}
  where $\tau_{ij}(t) = (t-t_{i,j-1})/(t_{ij}-t_{i,j-1})$ and
  $\{w_{n}^{[q_{ij}]}\}$ is a set of polynomial weight functions on $[0,1]$.
\end{lemma}
\begin{proof}
  The result follows from the definitions of the \mcgq{}, \mdgq{},
  \mcgqd{}, and \mdgqd{} methods, using an appropriate basis for the
  trial and test spaces. See \cite{logg:thesis:03} for details.\qquad
\end{proof}

\begin{theorem}[existence of solutions]
  Let $K=\max K_n$ be the maximum time slab length and define the Lipschitz constant
  $L_f>0$ by
  \begin{equation}
    \| f(x,t) - f(y,t) \|_{l_{\infty}} \leq L_f \| x - y \|_{l_{\infty}} \quad \forall t\in [0,T] \ \forall x,y \in \real^N.
  \end{equation}
  If now
  \begin{equation}
    K C L_f < 1,
  \end{equation}
  where $C = C(q) >0$ is a constant depending only on the order and method,
  the fixed point iteration \emph{(\ref{eq:explicit})} converges
  to the unique solution of
  \emph{(\ref{eq:fem,mcg})}, \emph{(\ref{eq:fem,mdg})}, \emph{(\ref{eq:mcgqd})}, and \emph{(\ref{eq:mdgqd})}, respectively.
\end{theorem}
\begin{proof}
  The result follows by Lemma \ref{lem:explicit} and an application of
  the Banach fixed point theorem. See \cite{logg:thesis:03} for details.\qquad
\end{proof}

\section{Stability of solutions}
\label{sec:stability}

Write the dual problem (\ref{eq:dual}) for $\phi=\phi(t)$ in
the form
\begin{equation} \label{eq:linear,dual}
	\begin{split}
		- \dot{\phi}(t) + A^{\top}(t) \phi(t) &= g, \quad t\in[0,T), \\
		\phi(T) &= \psi.
	\end{split}
\end{equation}
For simplicity, we consider only the case $g = 0$.
With $w(t) = \phi(T-t)$, we have $\dot{w}(t) = -\dot{\phi}(T-t) = - A^{\top}(T-t) w(t)$, and
so (\ref{eq:linear,dual}) can be written as a forward problem
for $w$ in the form
\begin{equation} \label{eq:linear,common}
	\begin{split}
	 \dot{w}(t) + B(t) w(t) &= 0, \quad t\in (0,T], \\
	 w(0) &= w_0,
	\end{split}
\end{equation}
where $w_0 = \psi$ and $B(t)=A^{\top}(T-t)$. Below, $w$ represents
either $u$ or $\phi(T-\cdot)$ and, correspondingly, $W$ represents
either the discrete $\mathrm{mc/dG}(q)$ approximation $U$ of $u$ or
the discrete $\mathrm{mc/dG}(q)^*$ approximation $\Phi$ of $\phi$.

\subsection{A general exponential estimate}
\label{sec:exponential}

The general exponential stability estimate is based on the following
version of the discrete Gronwall inequality.
\begin{lemma}[discrete Gronwall inequality]
  \label{lem:gronwall}
  Assume that $z,a : \naturals \rightarrow \real$ are nonnegative,
  $a(m)\leq 1/2$ for all $m$, and
  $z(n) \leq C + \sum_{m=1}^n a(m)z(m)$ for all $n$.
  Then
  $z(n) \leq 2C \exp( \sum_{m=1}^{n-1} 2a(m))$
  for $n = 1,2,\ldots$.
\end{lemma}
\begin{proof}
  By a standard discrete Gronwall inequality \cite{NiaPha00},
  {$z(n) \leq C \exp( \sum_{m=0}^{n-1} a(m) )$}
  if {$z(n) \leq$} {$ C + \sum_{m=0}^{n-1} a(m)z(m)$} for $n\geq 1$ and $z(0) \leq C$.
  Here, {$(1-a(n)) z(n) \leq C + \sum_{m=1}^{n-1}$} {$ a(m)z(m)$}, and so
  {$z(n) \leq 2C + \sum_{m=1}^{n-1} 2a(m) z(m)$},
  since $1 - a(n) \geq 1/2$. The result now follows if we take $a(0) = z(0) = 0$.\qquad
\end{proof}

\begin{theorem}[stability estimate]
  \label{th:estimate,exponential}
  Let $W$ be the \mcgq{}, \mdgq{}, \mcgqd{}, or \mdgqd{} solution of \emph{(\ref{eq:linear,common})}.
  Then  there is a constant $C = C(q)$, depending only on the highest order $\max q_{ij}$,
  such that if $K_n C \|B\|_{L_{\infty}([T_{n-1},T_n],l_p)} \leq 1$ for $n=1,\ldots,M$,
  then
  \begin{equation}
    \label{eq:stability,exponential}
    \|W\|_{L_{\infty}([T_{n-1},T_n],l_p)} \leq
    C \|w_0\|_{l_p} \exp\left( \sum_{m=1}^{n-1} K_m C \|B\|_{L_{\infty}([T_{m-1},T_m],l_p)} \right)
  \end{equation}
  for $n=1,\ldots,M$, $1\leq p \leq \infty$.
\end{theorem}

\begin{proof}
  By Lemma \ref{lem:explicit}, we can write the \mcgq, \mdgq, \mcgqd, and \mdgqd\ methods in the form
  \smash{$\xi_{ijn'} = w_i(0) +
   \int_0^{t_{i,j-1}} f_i(W,\cdot) \, dt +
   \int_{I_{ij}} w_{n'}^{[q_{ij}]}(\tau_{ij}(t)) f_i(W,\cdot) \, dt$}.\break
  Applied to the linear model problem (\ref{eq:linear,common}), we have
  \smash{$\xi_{ijn'} = w_i(0) -
   \int_0^{t_{i,j-1}} (BW)_i \, dt -$}\break
   \smash{$\int_{I_{ij}} w_{n'}^{[q_{ij}]}(\tau_{ij}(t)) (BW)_i \, dt$},
  and so
  \begin{eqnarray*}
    \vert \xi_{ijn'} \vert
    &\leq&
    \vert w_i(0) \vert +
    \bigg| \int_0^{t_{i,j-1}} (BW)_i \, dt \bigg| +
    \bigg| \int_{I_{ij}} w_{n'}^{[q_{ij}]}(\tau_{ij}(t)) (BW)_i \, dt \bigg| \\
    &\leq&
    \vert w_i(0) \vert +
    C \int_0^{t_{ij}} \vert (BW)_i \vert \, dt
    \leq
    \vert w_i(0) \vert +
    C \int_0^{T_n} \vert (BW)_i \vert \, dt,
  \end{eqnarray*}
  where $T_n$ is smallest synchronized time level for which $t_{ij} \leq T_n$.
  It now follows that for all $t \in [T_{n-1},T_n]$, we have
  $\vert W_i(t) \vert \leq C \vert w_i(0) \vert + C \int_0^{T_n} \vert (BW)_i \vert \, dt$, and so
  \begin{displaymath}
    \|W(t)\|_{l_p}
    \leq
    C \|w_0\|_{l_p} + C \int_0^{T_n} \|BW\|_{l_p} \, dt
    =
    C \|w_0\|_{l_p} + C \sum_{m=1}^n \int_{T_{m-1}}^{T_m} \|BW\|_{l_p} \, dt.
  \end{displaymath}
  The result now follows by letting $\bar{W}_n = \|W\|_{L_{\infty}([T_{n-1},T_n],l_p)}$.\qquad
\end{proof}

\begin{remark}
  See {\rm \cite{logg:thesis:03}} for an extension
  to multiadaptive time-stepping
  of the strong stability estimate Lemma {\rm 6.1} for parabolic problems in {\rm \cite{EriJoh91}}.
\end{remark}

\section{Interpolation estimates}
\label{sec:interpolation}

In this section, we introduce a pair of carefully chosen interpolants,
\smash{$\picgq$} and \smash{$\pidgq$}, which are central to  the a priori error
analysis of the \mcgq{} and \mdgq{} methods. The interpolants are
defined in section~\ref{sec:interpolation:interpolants}. In
section~\ref{sec:interpolation:basic}, we discuss the
interpolation of piecewise smooth functions, that is, the
interpolation of functions which may be discontinuous within the
interval of interpolation, and then present the basic general
interpolation estimates for the two interpolants \smash{$\picgq$} and \smash{$\pidgq$}.

For the a priori error analysis of the \mcgq{} and \mdgq{} methods, we
will also need a special interpolation estimate for the function
$\varphi = J^{\top} \Phi$, where $J$ is the Jacobian of the right-hand
side $f$ of (\ref{eq:u'=f}) and $\Phi$ is the discrete dual solution
as defined in section~\ref{sec:definition}, including estimates for
the size of the jump in function value and derivatives for the
function $\varphi$ at points of discontinuity.  These estimates are
proved in section~\ref{sec:interpolation:special}, based on a
representation formula for the \mcgq{} and \mdgq{} solutions of
(\ref{eq:u'=f}).

\subsection{Interpolants}
\label{sec:interpolation:interpolants}

The interpolant $\picgq: V \rightarrow \mathcal{P}^{q}([a,b])$ is
defined by the following conditions:
\begin{equation} \label{eq:interp,cg}
  \begin{split}
    & \picgq v(a) = v(a)\quad \text{ and }\quad \picgq v(b) = v(b), \\
    & \int_a^b \big(v - \picgq v\big) w \, dx = 0 \quad\forall w\in\mathcal{P}^{q-2}([a,b]),
  \end{split}
\end{equation}
where $V$ denotes the set of functions that are piecewise \smash{$\mathcal{C}^{q+1}$} and bounded on $[a,b]$.
In other words, \smash{$\picgq v$} is the polynomial of degree $q$ that interpolates $v$
at the end-points of the interval $[a,b]$ and additionally satisfies $q-1$ projection conditions.
This is illustrated in Figure \ref{fig:picg}.
We also define the dual interpolant \smash{$\picgqd$} as the standard $L_2$-projection onto \smash{$\mathcal{P}^{q-1}([a,b])$}.

\begin{figure}[htbp]
  \begin{center}
    \psfrag{x}{\small $x$}
    \psfrag{y}{\hspace{-0.7cm}\small $\picgq v$, $v$}
    \psfrag{0.5}{}
    \includegraphics[width=6cm]{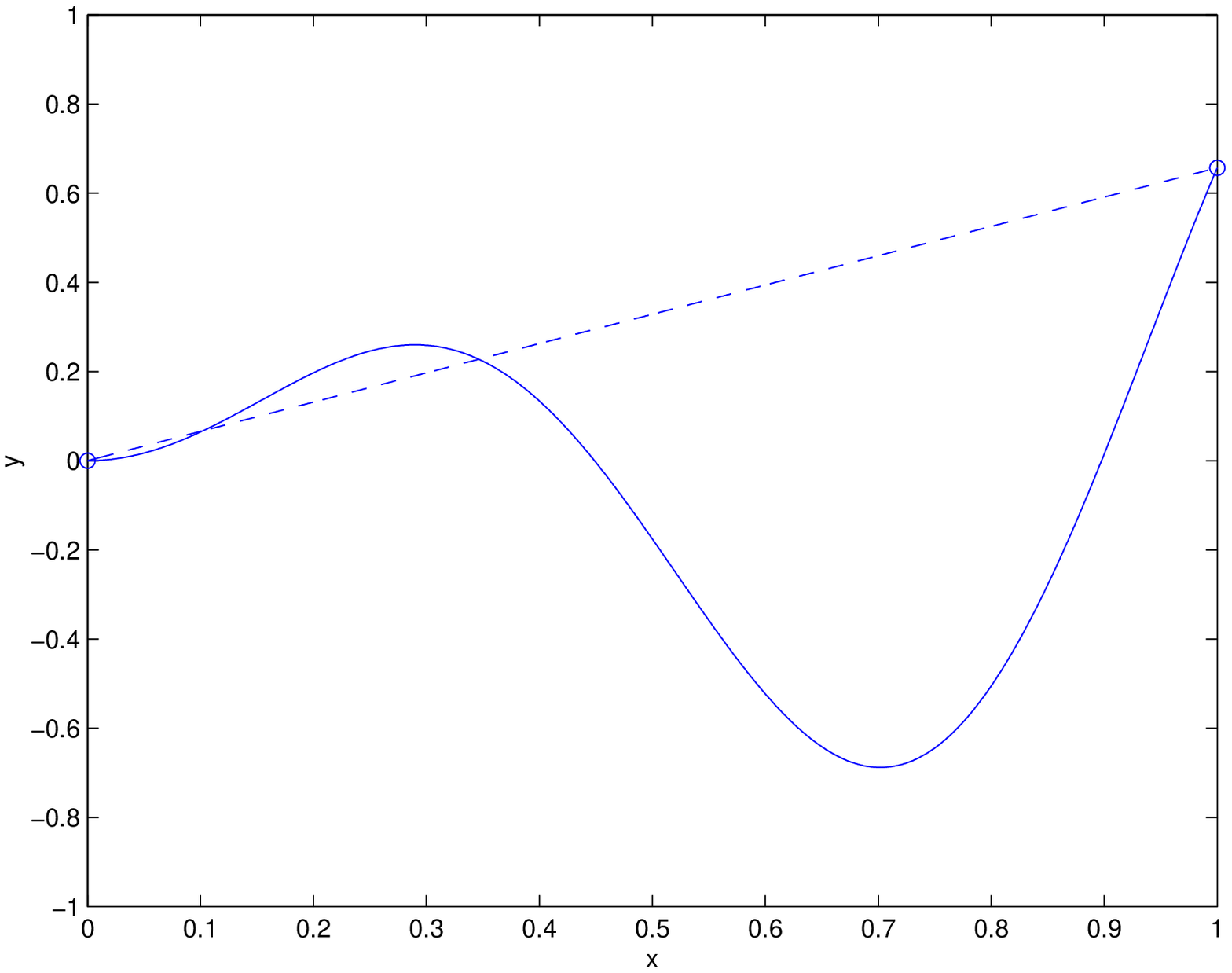}
    \psfrag{y}{}
    \includegraphics[width=6cm]{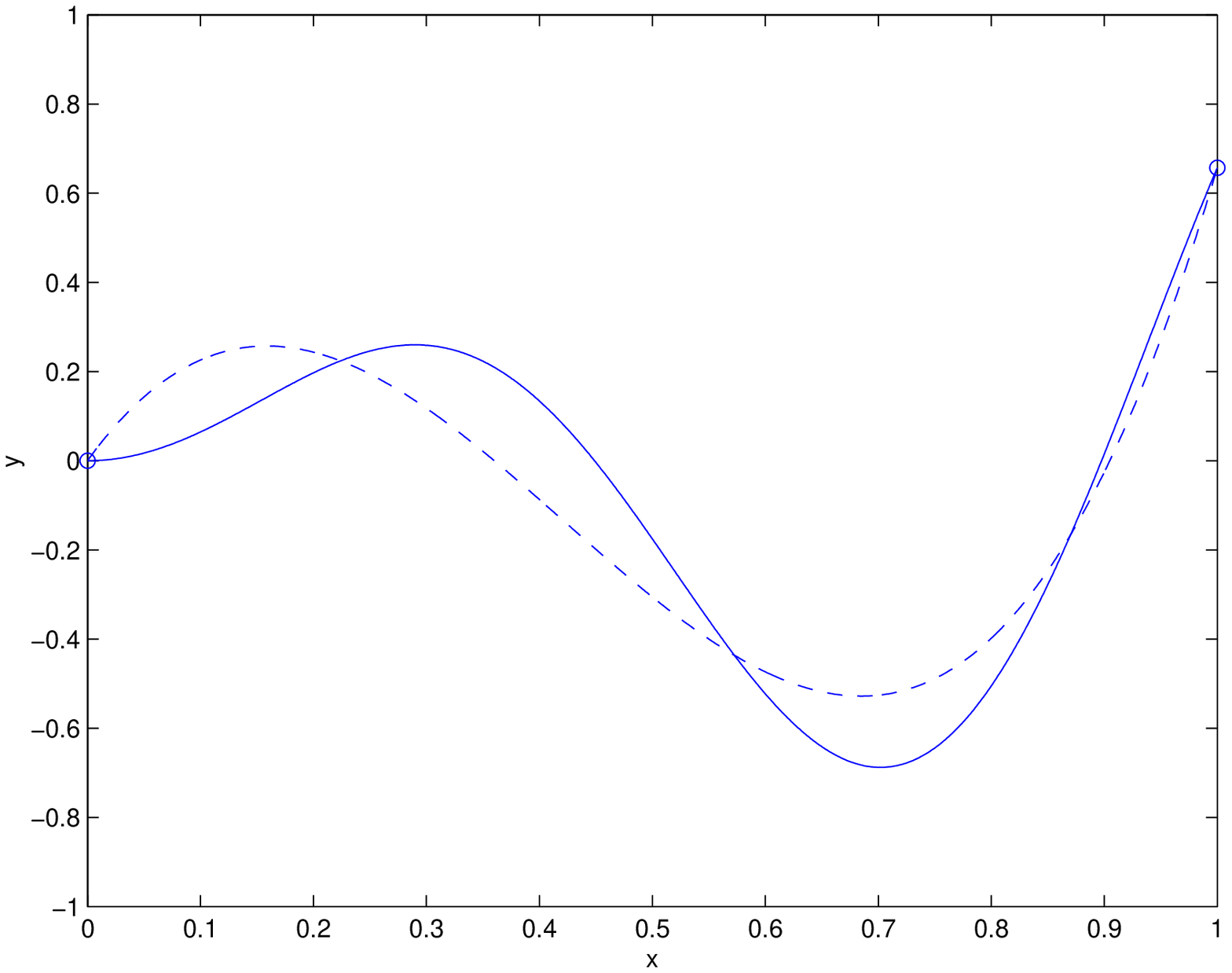}
    \caption{The interpolant $\picgq v$ (dashed) of the function $v(x) = x \, \sin(7x)$ (solid) on $[0,1]$ for $q=1$ (left) and $q=3$ (right).}
    \label{fig:picg}
  \end{center}
\end{figure}

\begin{figure}[t]
  \begin{center}
    \psfrag{x}{\small $x$}
    \psfrag{y}{\hspace{-0.7cm}\small $\pidgq v$, $v$}
    \psfrag{0.5}{}
    \includegraphics[width=6cm]{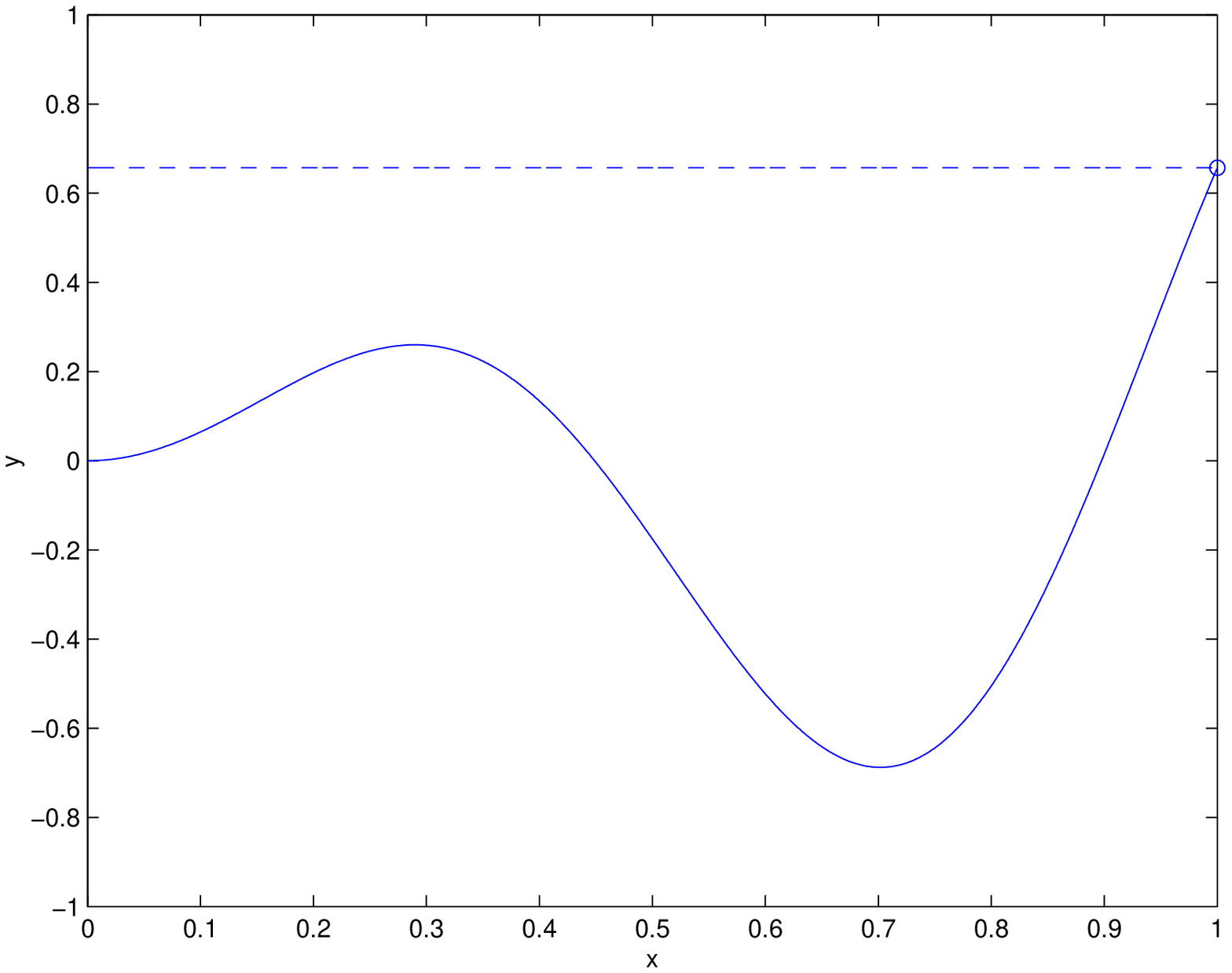}
    \psfrag{y}{}
    \includegraphics[width=6cm]{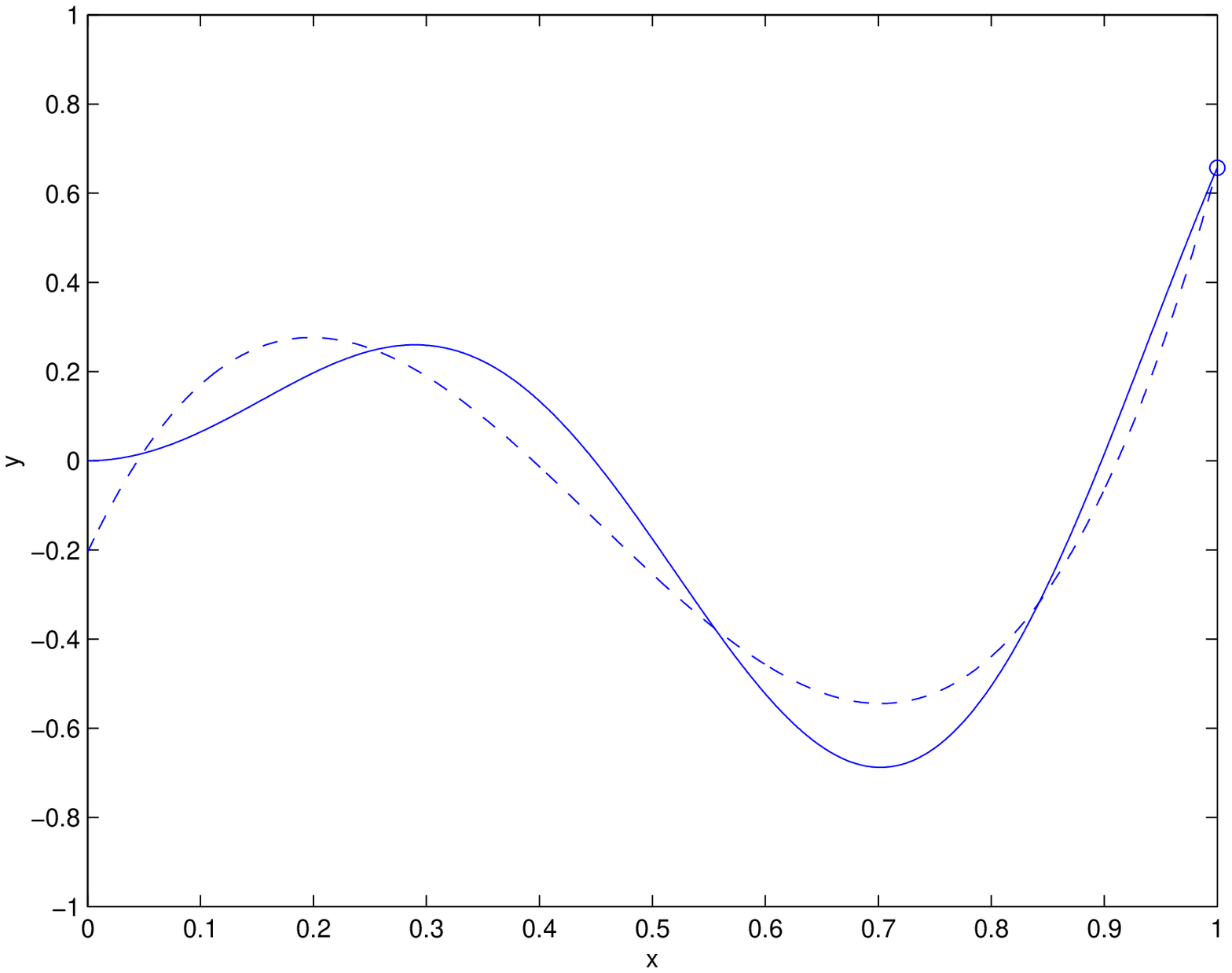}
    \caption{The interpolant $\pidgq v$ (dashed) of the function $v(x) = x \, \sin(7x)$ (solid) on $[0,1]$ for $q=0$ (left) and $q=3$ (right).}
    \label{fig:pidg}
  \end{center}
\end{figure}

The interpolant $\pidgq: V \rightarrow \mathcal{P}^{q}([a,b])$ is
defined by the following conditions:
\begin{equation} \label{eq:interp,dg}
  \begin{split}
    & \pidgq v(b) = v(b), \\
    & \int_a^b \big(v - \pidgq v\big) w \, dx = 0 \quad\forall w\in\mathcal{P}^{q-1}([a,b]),
  \end{split}
\end{equation}
that is, \smash{$\pidgq v$} is the polynomial of degree $q$ that interpolates
$v$ at the right end-point of the interval $[a,b]$ and additionally
satisfies $q$ projection conditions.  This is illustrated in Figure
\ref{fig:pidg}. The dual interpolant \smash{$\pidgqd$} is defined similarly,
with the difference being that the left end-point $x=a$ is used for
interpolation.

\subsection{Basic interpolation estimates}
\label{sec:interpolation:basic}

To estimate the size of the interpolation error $\pi v - v$ for a
given function $v$, we express the interpolation error in terms of the
regularity of $v$ and the length of the interpolation interval, $k = b - a$.
Specifically, when $v \in \mathcal{C}^{q+1}([a,b]) \subset V$ for some
$q\geq 0$, we obtain estimates of the form
\begin{equation}
  \label{eq:basic}
  \big\| (\pi v)^{(p)} - v^{(p)} \big\| \leq C k^{q+1-p} \big\| v^{(q+1)} \big\|, \quad p=0,\ldots,q+1,
\end{equation}
where $\| \cdot \| = \| \cdot \|_{L_{\infty}([a,b])}$ denotes the
maximum norm on $[a,b]$. This estimate is a simple consequence of the
Peano kernel theorem \cite{Pow88}  if one can show that the
interpolant $\pi : V \rightarrow \mathcal{P}^q([a,b]) \subset V$ is
linear and bounded on $V$ and that $\pi$ is exact
on $\mathcal{P}^q([a,b]) \subset V$, that is, $\pi v = v$ for all $v\in
\mathcal{P}^q([a,b])$.

In the general case, where the interpolated function $v$ is only
piecewise smooth (see Figure \ref{fig:piecewise_smooth}), we also need
to include the size of the jump $[v^{(p)}]_x$ in function value and
derivatives at each point $x$ of discontinuity within $(a,b)$ to
measure the regularity of the interpolated function $v$. In
\cite{logg:thesis:03}, we prove the following extensions of the basic
estimate (\ref{eq:basic}).

\begin{figure}[t]
  \begin{center}
    \psfrag{a}{\ $a$}
    \psfrag{b}{\ $b$}
    \psfrag{t1}{\ $x_1$}
    \psfrag{t2}{\ $x_2$}
    \psfrag{v}{$v$}
    \psfrag{pv}{$\pi v$}
    \includegraphics[width=10cm,height=4cm]{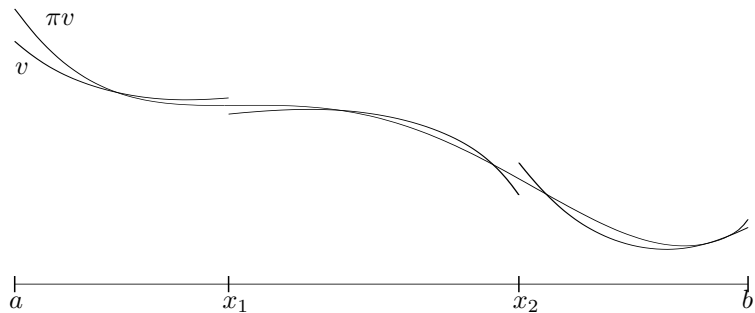}
    \caption{A piecewise smooth function $v$ and its interpolant $\pi v$.}
    \label{fig:piecewise_smooth}
  \end{center}
\end{figure}

\begin{lemma}
  \label{lem:estimate,piecewise}
  If $\pi$ is linear and bounded on $V$ and is exact on
  $\mathcal{P}^{q}([a,b]) \subset V$, then
  there is a constant $C=C(q)>0$ such that
  for all $v$ piecewise $\mathcal{C}^{q+1}$ on $[a,b]$ with discontinuities at $a<x_1<\cdots<x_n<b$,
\begin{equation}
    \big\| (\pi v)^{(p)} - v^{(p)} \big\| \leq C k^{r+1-p} \big\| v^{(r+1)} \big\| +
    C \sum_{j=1}^n \sum_{m=0}^r k^{m-p} \big| \big[ v^{(m)} \big]_{x_j} \big|
  \end{equation}
  for $p=0,\ldots,r+1$, $r=0,\ldots,q$.
\end{lemma}

\begin{lemma}
  \label{lem:derivatives,piecewise}
  If $\pi$ is linear and bounded on $V$ and is exact on
  $\mathcal{P}^{q}([a,b]) \subset V$,
  then there is a constant $C=C(q)>0$ such that
  for all $v$ piecewise $\mathcal{C}^{q+1}$ on $[a,b]$ with discontinuities at $a<x_1<\cdots<x_n<b$,
  \begin{equation}
    \label{eq:derivatives,piecewise}
    \big\| (\pi v)^{(p)} \big\|
    \leq
    C \big\| v^{(p)}\big\| +
    C \sum_{j=1}^n \sum_{m=0}^{p-1} k^{m-p} \big| \big[ v^{(m)} \big]_{x_j} \big|
  \end{equation}
  for $p=0,\ldots,q$.
\end{lemma}

Lemmas \ref{lem:estimate,piecewise} and
\ref{lem:derivatives,piecewise} apply to both the $\picgq$
interpolant (for $q \geq 1$) and the $\pidgq$
interpolant (for $q \geq 0$) defined in section
\ref{sec:interpolation:interpolants}. The linearity of both
interpolants follows directly from the definition of the
interpolants. The proofs that both interpolants are bounded and
exact on $\mathcal{P}^{q}([a,b])$ are given in detail in
\cite{logg:thesis:03} and \cite{logg:preprint:11}.

\subsection{A special interpolation estimate}
\label{sec:interpolation:special}

To prove a priori error estimates for \mcgq{} and \mdgq{} in
section \ref{sec:apriori}, we need to estimate the interpolation error
$\pi \varphi - \varphi$ for the function $\varphi$ defined by
\begin{equation}
  \label{eq:varphi}
  \varphi_i = (J^{\top}(\pi u,u,\cdot) \Phi)_i = \sum_{l=1}^N J_{li}(\pi u,u,\cdot) \Phi_l, \quad i=1,\ldots,N.
\end{equation}
We note that $\varphi_i$ may be discontinuous within $I_{ij}$ if $I_{ij}$ contains a node for some other
component, which is generally the case. This
is illustrated in Figure \ref{fig:varphi}. Note that on the right-hand
side $f$ is linearized around a mean value of $\pi u$ and $u$.

\begin{figure}
  \begin{center}
    \psfrag{Iij}{$I_{ij}$}
    \psfrag{tij-1}{$t_{i,j-1}$}
    \psfrag{ti}{$t_{ij}$}
    \psfrag{phii}{$\Phi_i(t)$}
    \psfrag{phin}{$\Phi_l(t)$}
    \includegraphics[width=10cm,height=6cm]{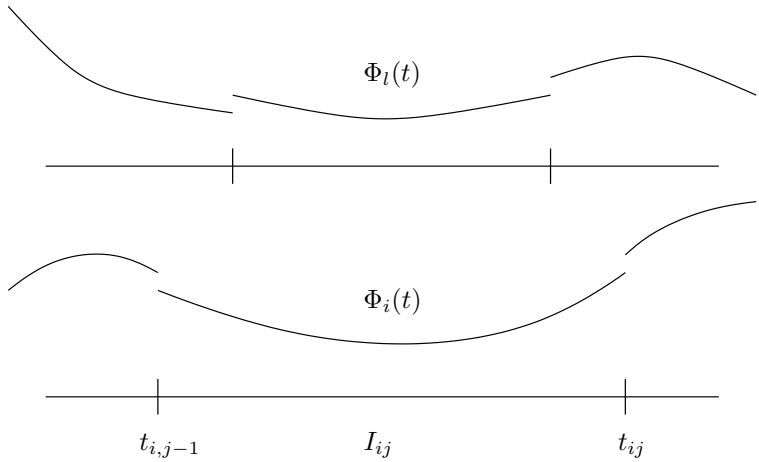}
    \caption{If some other component $l \neq i$ has a node within $I_{ij}$, then
      $\Phi_l$ may be discontinuous within $I_{ij}$, causing $\varphi_i$ to be discontinuous
      within $I_{ij}$.}
    \label{fig:varphi}
  \end{center}
\end{figure}

An interpolation estimate for $\pi \varphi - \varphi$ follows directly
from Lemma \ref{lem:estimate,piecewise}. To use this estimate, we need
to estimate the size of the jump in function value and derivatives at
each internal node $t_{ij}$ of the partition $\mathcal{T}$. To obtain
this estimate, we need to make a number of additional assumptions on
the right-hand side $f$ of (\ref{eq:u'=f}) and the partition
$\mathcal{T}$. These assumptions are discussed in section
\ref{sec:assumptions}. Based on the assumptions and the representation
formula presented in section \ref{sec:representation}, we obtain
the jump estimates in section \ref{sec:jumpestimates} and, finally, in section
\ref{sec:special,estimate}, the interpolation estimate for $\varphi$.

\subsubsection{A representation formula}
\label{sec:representation}

The proof of jump estimates for the multiadaptive Galerkin methods
\mcgq{} and \mdgq{} is based on expressing the solutions as certain
interpolants. These representations are obtained as follows. Let $U$
be the \mcgq\ or \mdgq\ solution of (\ref{eq:u'=f}) and define, for
$i=1,\ldots,N$,
\begin{equation} \label{eq:tildefunction}
  \tilde{U}_i(t) = u_i(0) + \int_0^t f_i(U(s),s) \, ds.
\end{equation}
Similarly, for $\Phi$ the \mcgqd\ or \mdgqd\ solution of (\ref{eq:dual}), we define, for
$i=1,\ldots,N$,
\begin{equation} \label{eq:tildefunction,dual}
  \tilde{\Phi}_i(t) = \psi_i + \int_t^T f_i^*(\Phi(s),s) \, ds,
\end{equation}
where $f^*(\Phi,\cdot) = J^{\top}(\pi u,U,\cdot)\Phi + g$.
We note that $\dot{\tilde{U}} = f(U,\cdot)$ and $-\dot{\tilde{\Phi}} = f^*(\Phi,\cdot)$.

It now turns out that $U$ can be expressed as an interpolant of
$\tilde{U}$. Similarly, $\Phi$ can be expressed as an interpolant of
$\tilde{\Phi}$. We present these representations in Lemmas
\ref{lem:representation,cg} and \ref{lem:representation,dg}.
We remind the reader about the interpolants $\picgq$, $\picgqd$,
$\pidgq$, and  $\pidgqd$, defined in section \ref{sec:interpolation:interpolants}.

\begin{lemma}
  \label{lem:representation,cg}
  The \mcgq\ solution $U$ of \emph{(\ref{eq:u'=f})}
  can expressed in the form \smash{$U = \picgq \tilde{U}$}.
  Similarly, the \mcgqd\ solution $\Phi$ of \emph{(\ref{eq:dual})} can be expressed in the form
  {$\Phi = \picgqd \tilde{\Phi}$},
  that is, {$U_i = \picg^{[q_{ij}]} \tilde{U}_i$} and
  {$\Phi_i = \pi_{\mathrm{cG}^*}^{[q_{ij}]} \tilde{\Phi}_i$} on each local interval $I_{ij}$.
\end{lemma}
\begin{proof}
  The representation formulas follow by the definitions of the
  \mcgq{} and \mcgqd{} methods and the interpolants {$\picgq$} and
  {$\picgqd{}$}. See \cite{logg:thesis:03} for details.\qquad
\end{proof}

\begin{lemma}
  \label{lem:representation,dg}
  The \mdgq\ solution $U$ of \emph{(\ref{eq:u'=f})}
  can expressed in the form \smash{$U = \pidgq \tilde{U}$}.
  Similarly, the \mdgqd\ solution $\Phi$ of \emph{(\ref{eq:dual})} can be expressed in the form
  {$\Phi = \pidgqd \tilde{\Phi}$},
  that is, {$U_i = \pidg^{[q_{ij}]} \tilde{U}_i$} and {$\Phi_i = \pidgd^{[q_{ij}]} \tilde{\Phi}_i$} on each local interval {$I_{ij}$}.
\end{lemma}
\begin{proof}
  The representation formulas follow by the definitions of the
  \mdgq{} and \mdgqd{} methods and the interpolants $\pidgq$ and
  $\pidgqd{}$. See \cite{logg:thesis:03} for details.\qquad
\end{proof}

\subsubsection{Assumptions}
\label{sec:assumptions}

To estimate the size of the jump in function value and derivatives
for the function $\varphi$ defined in (\ref{eq:varphi}), we make the
following assumptions. Given a time slab $\mathcal{T}$, assume that
for each pair of local intervals $I_{ij}$ and $I_{mn}$ within the time
slab, we have
\def\theequation{A1}
\begin{equation}
  q_{ij} = q_{mn} = \bar{q}
  \end{equation}
and
  \def\theequation{A2}
\begin{equation}
  k_{ij} > \alpha \ k_{mn}
\end{equation}
for some $\bar{q} \geq 0$ and some $\alpha \in (0,1)$. The dependence
on $\alpha$ in the error estimates is weak (see Remark \ref{rem:alpha}),
so assumption (A2) does not prevent multiadaptivity.

We also assume that the problem (\ref{eq:u'=f}) is autonomous,
 \def\theequation{A3}
\begin{equation}
  \partial f_i / \partial t = 0, \quad i=1,\ldots,N,
 \end{equation}
noting that the dual problem nevertheless will be nonautonomous in general.
Furthermore, we assume that
  \def\theequation{A4}
\begin{equation}
  \|f_i\|_{D^{\bar{q}+1}(\mathcal{T})} < \infty, \quad i=1,\ldots,N,
\end{equation}
where $\|\cdot\|_{D^p(\mathcal{T})}$ is defined for $v : \real^N \rightarrow \real$ and $p\geq 0$ by
$\|v\|_{D^p(\mathcal{T})} = \max_{n=0,\ldots,p} $ $\|D^n v\|_{L_{\infty}(\mathcal{T},l_{\infty})}$,
with the norm $\|D^n v\|_{L_{\infty}(\mathcal{T},l_{\infty})}$ defined by
$\|D^n v \, w^1 \cdots w^n\|_{L_{\infty}(\mathcal{T})} \leq $ $\|D^n
v\|_{L_{\infty}(\mathcal{T},l_{\infty})} \|w^1\|_{l_{\infty}} \cdots
\|w^n\|_{l_{\infty}}$ for all $w^1,\ldots,w^n \in \real^N$,
and $D^n v$ the $n$th-order tensor given by
\begin{displaymath}
  D^n v \, w^1 \cdots w^n = \sum_{i_1=1}^N \cdots \sum_{i_n=1}^N
  \frac{\partial^n v}{\partial x_{i_1} \cdots \partial x_{i_n}} \, w^1_{i_1} \cdots w^n_{i_n}.
\end{displaymath}
Furthermore, we choose $C_f \geq \max_{i=1,\ldots,N} \|f_i\|_{D^{\bar{q}+1}(\mathcal{T})}$ such that
\addtocounter{equation}{-4}%
\renewcommand{\theequation}{\arabic{section}.\arabic{equation}}
\begin{equation}
  \|d^p/dt^p (\partial f/ \partial u)^{\top}(x(t))\|_{l_{\infty}} \leq
  C_f C_x^p
\end{equation}
for $p=0,\ldots,\bar{q}$, and
\begin{equation}
  \big\|[d^p/dt^p (\partial f/ \partial u)^{\top}(x(t))]_t\big\|_{l_{\infty}} \leq C_f \sum_{n=0}^p C_x^{p-n} \big\|\big[x^{(n)}\big]_t\big\|_{l_{\infty}}
\end{equation}
for $p=0,\ldots,\bar{q}-1$ and any given $x:\real\rightarrow\real^N$,
where $C_x > 0$ denotes a constant such that
$\|x^{(n)}\|_{L_\infty(\mathcal{T},l_{\infty})} \leq C_x^n$ for
$n=1,\ldots,p$. Note that $C_f = C_f(t)$ defines a piecewise constant
function on the partition $0=T_0<T_1<\cdots<T_M=T$.
Note also that assumption (A4) implies that each $f_i$ is bounded by $C_f$.

We further assume that there is a constant $c_k>0$ such that
 \def\theequation{A5}
\begin{equation}
  k_{ij} C_f \leq c_k
 \end{equation}
for each local interval $I_{ij}$.
We summarize the list of assumptions as follows:
\newcounter{lcount}
\begin{list}{(A\arabic{lcount})}{\usecounter{lcount} \setcounter{lcount}{0}}
\item
  the local orders $q_{ij}$ are equal within each time slab;
\item
  the local time steps $k_{ij}$ are semiuniform within each time slab;
\item
  $f$ is autonomous;
\item
  $f$ and its derivatives are bounded;
\item
  the local time steps $k_{ij}$ are small.
\end{list}

\subsubsection{Estimates of derivatives and jumps}
\label{sec:jumpestimates}

To estimate higher-order derivatives, we face the problem of taking
higher-order derivatives of $f(U(t),t)$ with respect to $t$. In Lemmas
\ref{lem:estimate,derivative} and \ref{lem:estimate,jump}, we
present basic estimates for composite functions $v \circ x$ with $v :
\real^N \rightarrow \real$ and $x : \real \rightarrow \real^N$. The
proofs are based on a straightforward application of the chain rule and
Leibniz rule and are given in full detail in \cite{logg:thesis:03}.
  \addtocounter{equation}{-1}%
\renewcommand{\theequation}{\arabic{section}.\arabic{equation}}

\begin{lemma}
  \label{lem:estimate,derivative}
  Let $v : \real^N \rightarrow \real$ be $p\geq 0$ times differentiable in all its variables,
  let $x : \real \rightarrow \real^N$ be $p$ times differentiable,
  and let $C_x > 0$ be a constant such that $\|x^{(n)}\|_{L_\infty(\real,l_{\infty})} \leq C_x^n$ for $n=1,\ldots,p$.
  Then there is a constant $C = C(p) > 0$ such that
  \begin{equation} \label{eq:estimate,1}
    \left\|\Od{p}{(v \circ x)}{t}\right\|_{L_{\infty}(\real)} \leq C \|v\|_{D^p(\real)} C_x^p.
  \end{equation}
\end{lemma}

\begin{lemma}
  \label{lem:estimate,jump}
  Let $v : \real^N \rightarrow \real$ be $p + 1 \geq 1$ times differentiable in all its variables,
  let $x : \real \rightarrow \real^N$ be $p$ times differentiable, except possibly at some $t\in\real$,
  and let $C_x > 0$ be a constant such that $\| x^{(n)} \|_{L_\infty(\real,l_{\infty})} \leq C_x^n$ for $n=1,\ldots,p$.
  Then there is a constant $C = C(p) > 0$ such that
  \begin{equation} \label{eq:estimate,2}
    \left| \left[ \Od{p}{(v \circ x)}{t} \right]_{t} \right|
    \leq
    C \|v\|_{D^{p+1}(\real)}  \sum_{n=0}^p C_x^{p-n} \big\|\big[x^{(n)}\big]_{t}\big\|_{l_{\infty}}.
  \end{equation}
\end{lemma}\unskip

We now prove estimates for derivatives and jumps of the \mcgq{} or
\mdgq{} solution $U$ of the general nonlinear problem (\ref{eq:u'=f}),
under the assumptions listed in section
\ref{sec:assumptions}. Similarly, one can obtain estimates for the
discrete dual solution $\Phi$ and the function $\varphi$ defined in
(\ref{eq:varphi}), from which the desired interpolation estimates
follow.

To obtain estimates for the multiadaptive solution $U$, we first
prove estimates for the function $\tilde{U}$ defined in section
\ref{sec:representation}. The estimates for $U$ then follow by induction.

To simplify the estimates, we introduce the following notation. For
given $p>0$, let $C_{U,p}\geq C_f$ be a constant such that
\begin{equation} \label{eq:CU}
  \big\|U^{(n)}\big\|_{L_{\infty}(\mathcal{T},l_{\infty})} \leq C_{U,p}^n, \quad n=1,\ldots,p.
\end{equation}
For $p=0$, we define $C_{U,0} = C_f$.
Temporarily, we assume that there is a constant $c_k'>0$ such that for each $p$,
  \def\theequation{A5$'$}
\begin{equation}
  k_{ij} C_{U,p} \leq c_k'.
\end{equation}
This assumption will be removed in Lemma \ref{lem:U,derivatives}.
In the following lemma, we use assumptions (A1), (A3), and (A4) to derive estimates for $\tilde{U}$
in terms of $C_{U,p}$ and $C_f$.
  \addtocounter{equation}{-1}%
\renewcommand{\theequation}{\arabic{section}.\arabic{equation}}

\begin{lemma}[derivative and jump estimates for $\tilde{U}$]\label{lem:utilde}
  Let $U$ be the \mcgq\ or \mdgq\ solution of \emph{(\ref{eq:u'=f})} and define $\tilde{U}$ as in \emph{(\ref{eq:tildefunction})}.
  If assumptions \emph{(A1)}, \emph{(A3)}, and \emph{(A4)} hold, then there is a constant $C=C(\bar{q})>0$ such that
  \begin{equation} \label{eq:tilde,derivatives}
    \big\|\tilde{U}^{(p)}\big\|_{L_{\infty}(\mathcal{T},l_{\infty})} \leq C C_{U,p-1}^{p}, \quad p=1,\ldots,\bar{q}+1,
  \end{equation}
  and
  \begin{equation} \label{eq:tilde,jumps}
    \big\|\big[\tilde{U}^{(p)}\big]_{t_{i,j-1}}\big\|_{l_{\infty}}
    \leq
    C \sum_{n=0}^{p-1} C_{U,p-1}^{p-n} \big\|\big[U^{(n)}\big]_{t_{i,j-1}}\big\|_{l_{\infty}}, \quad p=1,\ldots,\bar{q}+1,
  \end{equation}
  for each local interval $I_{ij}$, where $t_{i,j-1}$ is an internal node of the time slab $\mathcal{T}$.
\end{lemma}
\begin{proof}
  By definition, $\tilde{U}_i^{(p)} = \Od{p-1}{}{t} f_i(U)$, and so the results follow directly by Lemmas \ref{lem:estimate,derivative}
  and \ref{lem:estimate,jump},
  noting that $C_f \leq C_{U,p-1}$.\qquad
\end{proof}

By Lemma \ref{lem:utilde}, we now obtain the following estimate for the
size of the jump in function value and derivatives for $U$.

\begin{lemma}[jump estimates for $U$] \label{lem:U,jumps}
  Let $U$ be the \mcgq\ or \mdgq\ solution of \emph{(\ref{eq:u'=f})}.
  If assumptions \emph{(A1)}--\emph{(A5)} and \emph{(A5$'$)} hold, then there is a constant $C=C(\bar{q},c_k,c_k',\alpha)>0$ such that
  \begin{equation}
    \label{eq:u,jumps,first}
    \big\|\big[U^{(p)}\big]_{t_{i,j-1}}\big\|_{l_{\infty}} \leq
    C k_{ij}^{r+1-p} C_{U,r}^{r+1}, \quad p=0,\ldots,r+1, \quad r = 0,\ldots,\bar{q},
  \end{equation}
  for each local interval $I_{ij}$, where $t_{i,j-1}$ is an internal node of the time slab $\mathcal{T}$.
\end{lemma}
\begin{proof}
  The proof is by induction.
  We first note that at $t=t_{i,j-1}$, we have
  \begin{eqnarray*}
    \big[U_i^{(p)}\big]_{t}
    &=&  U_i^{(p)}(t^+) - \tilde{U}_i^{(p)}(t^+) +
    \tilde{U}_i^{(p)}(t^+) - \tilde{U}_i^{(p)}(t^-) +
    \tilde{U}_i^{(p)}(t^-)-U_i^{(p)}(t^-) \\
    &\equiv& e_+ + e_0 + e_-.
  \end{eqnarray*}
  By Lemma \ref{lem:representation,cg} (or Lemma \ref{lem:representation,dg}), $U$ is an interpolant of $\tilde{U}$
  and so, by Lemma \ref{lem:estimate,piecewise}, we have
  \begin{displaymath}
    \vert e_+ \vert
    \leq
    C k_{ij}^{r+1-p} \big\| \tilde{U}_i^{(r+1)} \big\|_{L_{\infty}(I_{ij})} +
    C \sum_{x\in \mathcal{N}_{ij}} \sum_{m=1}^r k_{ij}^{m-p} \big\vert \big[ \tilde{U}_i^{(m)} \big]_{x} \big\vert
  \end{displaymath}
  for $p=0,\ldots,r+1$ and $r=0,\ldots,\bar{q}$.
  Note that the second sum starts at $m=1$ rather than at $m=0$, since $\tilde{U}$ is continuous.
  Similarly, we have
  \begin{displaymath}
    \vert e_- \vert
    \leq
    C k_{i,j-1}^{r+1-p} \big\| \tilde{U}_i^{(r+1)} \big\|_{L_{\infty}(I_{i,j-1})} +
    C \sum_{x\in \mathcal{N}_{i,j-1}} \sum_{m=1}^r k_{i,j-1}^{m-p} \big\vert \big[ \tilde{U}_i^{(m)} \big]_{x} \big\vert.
  \end{displaymath}
  To estimate $e_0$, we note that $e_0=0$ for $p=0$, since $\tilde{U}$ is continuous.
  For $p=1,\ldots,\bar{q}+1$, Lemma \ref{lem:utilde} gives
  $\vert e_0 \vert
    =
    \vert [ \tilde{U}_i^{(p)} ]_t \vert
    \leq
    C \sum_{n=0}^{p-1} C_{U,p-1}^{p-n} \|[U^{(n)}]_t\|_{l_{\infty}}$.
  By assumption (A2), it then follows that
  (\ref{eq:u,jumps,first}) holds for $r=0$.

  Assume now that (\ref{eq:u,jumps,first}) holds for $r=\bar{r}-1\geq 0$.
  Then, by Lemma \ref{lem:utilde} and assumption (A5$'$), it follows that
  \begin{displaymath}
    \begin{split}
      \vert e_+ \vert
      &\leq
      C k_{ij}^{\bar{r}+1-p} C_{U,\bar{r}}^{\bar{r}+1} +
      C \sum_{x\in\mathcal{N}_{ij}} \sum_{m=1}^{\bar{r}} k_{ij}^{m-p}
      \sum_{n=0}^{m-1} C_{U,m-1}^{m-n} \big\|[U^{n}]_{x}\big\|_{l_{\infty}} \\
      &\leq
      C k_{ij}^{\bar{r}+1-p} C_{U,\bar{r}}^{\bar{r}+1} +
      C \sum k_{ij}^{m-p} C_{U,m-1}^{m-n} k_{ij}^{(\bar{r}-1)+1-n} C_{U,\bar{r}-1}^{(\bar{r}-1)+1} \\
      &\leq
      C k_{ij}^{\bar{r}+1-p} C_{U,\bar{r}}^{\bar{r}+1}
      \left( 1 +
      \sum (k_{ij} C_{U,\bar{r}-1})^{m-1-n}
      \right)
      \leq
      C k_{ij}^{\bar{r}+1-p} C_{U,\bar{r}}^{\bar{r}+1}.
      \end{split}
  \end{displaymath}
  Similarly, we obtain the estimate $\vert e_- \vert \leq C k_{ij}^{\bar{r}+1-p} C_{U,\bar{r}}^{\bar{r}+1}$.
  Finally, we use Lemma \ref{lem:utilde} and assumption (A5$'$) to obtain the estimate
  \begin{eqnarray*}
    \vert e_0 \vert
    &\leq&
    C \sum_{n=0}^{p-1} C_{U,p-1}^{p-n} \big\|[U^{n}]_t\big\|_{l_{\infty}}
    \leq
    C \sum_{n=0}^{p-1} C_{U,p-1}^{p-n} k_{ij}^{(\bar{r}-1)+1-n} C_{U,\bar{r}-1}^{(\bar{r}-1)+1} \\
    &=&
    C k_{ij}^{\bar{r}+1-p} C_{U,\bar{r}}^{\bar{r}+1} \sum_{n=0}^{p-1}  (k_{ij} C_{U,\bar{r}})^{p-1-n}
    \leq
    C k_{ij}^{\bar{r}+1-p} C_{U,\bar{r}}^{\bar{r}+1}.
  \end{eqnarray*}
  Summing up, we
  thus obtain $\vert [U_i^{(p)}]_t \vert \leq |e_+| + |e_0| + |e_-| \leq C k_{ij}^{\bar{r}+1-p} C_{U,\bar{r}}^{\bar{r}+1}$,
  and so (\ref{eq:u,jumps,first}) follows by induction.\qquad
\end{proof}

By Lemmas \ref{lem:utilde} and \ref{lem:U,jumps}, we now obtain the following estimate
for derivatives of the solution $U$.

\begin{lemma}[derivative estimates for $U$]
  \label{lem:U,derivatives}
  Let $U$ be the \mcgq\ or \mdgq\ solution of \emph{(\ref{eq:u'=f})}.
  If assumptions \emph{(A1)}--\emph{(A5)} hold, then there is a constant $C=C(\bar{q},c_k,\alpha)>0$ such that
  \begin{equation} \label{eq:U,derivatives}
    \big\|U^{(p)}\big\|_{L_{\infty}(\mathcal{T},l_{\infty})} \leq C C_f^p, \quad p=1,\ldots,\bar{q}.
  \end{equation}
\end{lemma}\unskip

\begin{proof}
  By Lemma \ref{lem:representation,cg} (or Lemma
  \ref{lem:representation,dg}), $U$ is an interpolant of $\tilde{U}$
  and so, by Lemma \ref{lem:estimate,piecewise}, we have
  \begin{displaymath}
    \big\| U_i^{(p)} \big\|_{L_{\infty}(I_{ij})} = \big\| (\pi \tilde{U}_i)^{(p)} \big\|_{L_{\infty}(I_{ij})} \leq
    C' \big\| \tilde{U}^{(p)}_i \big\|_{L_{\infty}(I_{ij})} +
    C' \sum_{x\in\mathcal{N}_{ij}} \sum_{m=1}^{p-1} k_{ij}^{m-p} \big\vert \big[ \tilde{U}^{(m)}_i \big]_x \big\vert
  \end{displaymath}
  for some constant $C'=C'(\bar{q})$.
  For $p=1$, we thus obtain the estimate
  \begin{displaymath}
    \| \dot{U}_i \|_{L_{\infty}(I_{ij})} \leq C' \| \dot{\tilde{U}}_i \|_{L_{\infty}(I_{ij})} = C' \| f_i(U) \|_{L_{\infty}(I_{ij})} \leq C' C_f
  \end{displaymath}
  by assumption (A4), and so (\ref{eq:U,derivatives}) holds for $p=1$.

  For $p=2,\ldots,\bar{q}$, assuming that (A5$'$) holds for $C_{U,p-1}$, we use Lemmas \ref{lem:utilde} and \ref{lem:U,jumps}
(with $r=p-1$) and assumption (A2) to obtain
  \begin{displaymath}
    \begin{split}
      \big\| U_i^{(p)}\big\|_{L_{\infty}(I_{ij})}
      &\leq
      C C_{U,p-1}^{p} +
      C \sum_{x\in\mathcal{N}_{ij}} \sum_{m=1}^{p-1} k_{ij}^{m-p} \sum_{n=0}^{m-1} C_{U,m-1}^{m-n} \big\|\big[U^{(n)}\big]_x\big\|_{l_{\infty}} \\
      &\leq
      C C_{U,p-1}^{p} +
      C \sum k_{ij}^{m-p} C_{U,m-1}^{m-n} k_{ij}^{(p-1)+1-n} C_{U,p-1}^{(p-1)+1} \\
      &\leq
      C C_{U,p-1}^{p} \left( 1 + \sum (k_{ij} C_{U,m-1})^{m-n} \right)
      \leq
      C C_{U,p-1}^{p},
    \end{split}
  \end{displaymath}
  where $C=C(\bar{q},c_k,c_k',\alpha)$.
  This holds for all components $i$ and all local intervals $I_{ij}$
  within the time slab $\mathcal{T}$, and so
  \begin{displaymath}
    \big\|U^{(p)}\big\|_{L_{\infty}(\mathcal{T},l_{\infty})} \leq C C_{U,p-1}^{p}, \quad p=1,\ldots,\bar{q},
  \end{displaymath}
  where by definition $C_{U,p-1}$ is a constant such that
  \smash{$\|U^{(n)}\|_{L_{\infty}(\mathcal{T},l_{\infty})} \leq C_{U,p-1}^n$} for $n=1,\ldots,p-1$.
  Starting at $p=1$, we now define $C_{U,1} = C_1 C_f$ with $C_1=C'=C'(\bar{q})$.
  It then follows that (A5$'$) holds for $C_{U,1}$ with $c_k'=C'c_k$, and thus
  \begin{displaymath}
    \big\|U^{(2)}\big\|_{L_{\infty}(\mathcal{T},l_{\infty})} \leq C C_{U,2-1}^{2} = C C_{U,1}^2 \equiv C_2 C_f^2,
  \end{displaymath}
  where $C_2 = C_2(\bar{q},c_k,\alpha)$. We may thus define $C_{U,2} = \max(C_1 C_f, \sqrt{C_2} C_f)$.
  Continuing, we note that (A5$'$) holds for $C_{U,2}$, and thus
  \begin{displaymath}
    \big\|U^{(3)}\big\|_{L_{\infty}(\mathcal{T},l_{\infty})} \leq C C_{U,3-1}^{3} = C C_{U,2}^3 \equiv C_3 C_f^3,
  \end{displaymath}
  where $C_3 = C_3(\bar{q},c_k,\alpha)$.
  In this way, we obtain a sequence of constants $C_1,\ldots,C_{\bar{q}}$, depending only
  on $\bar{q}$, $c_k$, and $\alpha$, such that
  $\|U^{(p)}\|_{L_{\infty}(\mathcal{T},l_{\infty})} \leq C_p C_f^p$ for $p=1,\ldots,\bar{q}$,
  and so (\ref{eq:U,derivatives}) follows if we take $C=\max_{i=1,\ldots,\bar{q}} C_i$.\qquad
\end{proof}

Having now removed the additional assumption (A5$'$), we obtain the following version of Lemma \ref{lem:U,jumps}.
\begin{lemma}[jump estimates for $U$] \label{lem:U,jumps,new}
  Let $U$ be the \mcgq\ or \mdgq\ solution of \emph{(\ref{eq:u'=f})}.
  If assumptions \emph{(A1)}--\emph{(A5)} hold, then there is a constant $C=C(\bar{q},c_k,\alpha)>0$ such that
  \begin{equation} \label{eq:U,jumps,new}
    \big\|\big[U^{(p)}\big]_{t_{i,j-1}}\big\|_{l_{\infty}} \leq
    C k_{ij}^{\bar{q}+1-p} C_f^{\bar{q}+1}, \quad p=0,\ldots,\bar{q},
  \end{equation}
  for each local interval $I_{ij}$, where $t_{i,j-1}$ is an internal node of the time slab $\mathcal{T}$.
\end{lemma}

Similarly, we obtain estimates for the discrete dual solution $\Phi$ and
the function $\varphi$. In Lemma \ref{lem:varphi}, we present the
estimates for the function $\varphi$.

\begin{lemma}[estimates for $\varphi$] \label{lem:varphi}
  Let $\varphi$ be defined as in \emph{(\ref{eq:varphi})}.
  If assumptions \emph{(A1)}--\emph{(A5)} hold, then there is a constant $C=C(\bar{q},c_k,\alpha)>0$ such that
  \begin{equation} \label{eq:varphi,1}
    \big\| \varphi_i^{(p)} \big\|_{L_{\infty}(I_{ij})}
    \leq
    C C_f^{p+1} \|\Phi\|_{L_{\infty}(\mathcal{T},l_{\infty})}, \quad p=0,\ldots,q_{ij},
  \end{equation}
  and
  \begin{equation} \label{eq:varphi,2}
    \big| \big[ \varphi_i^{(p)} \big]_x \big|
    \leq
    C k_{ij}^{r_{ij}-p} C_f^{r_{ij}+1} \|\Phi\|_{L_{\infty}(\mathcal{T},l_{\infty})}
    \quad \forall x\in \mathcal{N}_{ij}, \quad p=0,\ldots,q_{ij}-1,
  \end{equation}
  with $r_{ij} = q_{ij}$ for the \mcgq{} method and $r_{ij} = q_{ij} + 1$ for the \mdgq{} method.
  This holds for each local interval $I_{ij}$ within the time slab $\mathcal{T}$.
\end{lemma}

\subsubsection{Interpolation estimates}
\label{sec:special,estimate}

Using the basic interpolation estimate of section
\ref{sec:interpolation:basic}, we now obtain the following important
interpolation estimates for the function $\varphi$.

\begin{lemma}[interpolation estimates for $\varphi$] \label{lem:interpolationestimate,phi}
  Let $\varphi$ be defined as in \emph{(\ref{eq:varphi})}.
  If assumptions \emph{(A1)}--\emph{(A5)} hold, then there is a constant $C=C(\bar{q},c_k,\alpha)>0$ such that
  \begin{equation} \label{eq:varphi,estimate,cg}
    \big\| \picg^{[q_{ij}-2]} \varphi_i - \varphi_i \big\|_{L_{\infty}(I_{ij})}
    \leq
    C k_{ij}^{q_{ij}-1} C_f^{q_{ij}} \|\Phi\|_{L_{\infty}(\mathcal{T},l_{\infty})}, \quad q_{ij} = \bar{q} \geq 2,
  \end{equation}
  and
  \begin{equation} \label{eq:varphi,estimate,dg}
    \big\| \pidg^{[q_{ij}-1]} \varphi_i - \varphi_i \big\|_{L_{\infty}(I_{ij})}
    \leq
    C k_{ij}^{q_{ij}} C_f^{q_{ij}+1} \|\Phi\|_{L_{\infty}(\mathcal{T},l_{\infty})}, \quad q_{ij} = \bar{q} \geq 1,
  \end{equation}
  for each local interval $I_{ij}$ within the time slab $\mathcal{T}$.
\end{lemma}
\begin{proof}
  To prove (\ref{eq:varphi,estimate,cg}), we use
  Lemma \ref{lem:estimate,piecewise}, with $r=q_{ij}-2$ and $p=0$, together with Lemma \ref{lem:varphi}, to obtain
  \begin{displaymath}
    \begin{split}
      &\big\| \picg^{[q_{ij}-2]} \varphi_i - \varphi_i \big\|_{L_{\infty}(I_{ij})}
      \leq
      C k_{ij}^{q_{ij}-1} \big\| \varphi_i^{(q_{ij}-1)} \big\|_{L_{\infty}(I_{ij})} +
      C \sum_{x\in \mathcal{N}_{ij}} \sum_{m=0}^{q_{ij}-2} k_{ij}^{m} \big\vert \big[ \varphi_i^{(m)} \big]_x \big\vert \\
      &\quad \leq
      C k_{ij}^{q_{ij}-1} C_f^{q_{ij}} \|\Phi\|_{L_{\infty}(\mathcal{T},l_{\infty})} +
      C \sum_{x\in \mathcal{N}_{ij}} \sum_{m=0}^{q_{ij}-2} k_{ij}^{m}
      k_{ij}^{q_{ij}-m} C_f^{q_{ij}+1} \|\Phi\|_{L_{\infty}(\mathcal{T},l_{\infty})} \\
      &\quad =
      C k_{ij}^{q_{ij}-1} C_f^{q_{ij}} \|\Phi\|_{L_{\infty}(\mathcal{T},l_{\infty})} +
      C
      k_{ij}^{q_{ij}} C_f^{q_{ij}+1} \|\Phi\|_{L_{\infty}(\mathcal{T},l_{\infty})},
    \end{split}
  \end{displaymath}
  from which the estimate follows.
  The estimate for $\pidg^{[q_{ij}-1]} \varphi_i - \varphi_i$ is obtained similarly.~\qquad
\end{proof}

\begin{remark}
  \label{rem:alpha}
  Note that there is only a weak dependence on $c_k$ and $\alpha$,
  since the jump term contains an extra factor $k_{ij}$. If higher-order
  terms can be ignored, then the dependence on $c_k$ and $\alpha$ can
  be removed.
\end{remark}

\section{A priori error estimates}
\label{sec:apriori}

To prove a priori error estimates for the \mcgq{} and \mdgq{} methods,
we derive error representations in section
\ref{sec:errorrepresentation} and then obtain the a priori error
estimates in section~\ref{sec:aprioriestimates} for the general
nonlinear case. We refer to \cite{logg:thesis:03} for a sharp
a priori error estimate in the case of a parabolic model problem.

\subsection{Error representation}
\label{sec:errorrepresentation}

For each of the two methods, \mcgq\ and \mdgq, we represent the error in
terms of the discrete dual solution $\Phi$ and an interpolant $\pi u$
of the exact solution $u$ of (\ref{eq:u'=f}), using the special interpolants
\smash{$\pi u = \picgq u$} or \smash{$\pi u = \pidgq u$} defined in section \ref{sec:interpolation}.

We write the error $e = U - u$ in the form
\begin{equation}
  e = \bar{e} + (\pi u - u),
\end{equation}
where $\bar{e} \equiv U - \pi u$ is represented in terms of the
discrete dual solution and the residual of the interpolant. An
estimate for the second part of the error, $\pi u - u$, follows
directly from an interpolation estimate.

In Lemma \ref{lem:errorrepresentation,cg}, we present the error representation for
the \mcgq\ method, and then present the corresponding representation for
the \mdgq\ method in Lemma \ref{lem:errorrepresentation,dg}.
The error representations are obtained directly by choosing $\bar{e}$
as a test function for the discrete dual problems (\ref{eq:mcgqd})
and (\ref{eq:mdgqd}).

\begin{lemma}[error representation for \mcgq]
  \label{lem:errorrepresentation,cg}
  Let $U$ be the \mcgq\ solution of \emph{(\ref{eq:u'=f})},
  let $\Phi$ be the corresponding \mcgqd\ solution of the dual problem \emph{(\ref{eq:dual})}, and
  let $\pi u$ be any trial space approximation of the exact solution $u$ of \emph{(\ref{eq:u'=f})} that
  interpolates $u$ at the end-points of every local interval.
  Then
  \begin{displaymath}
    L_{\psi,g}(\bar{e})
    \equiv
    (\bar{e}(T),\psi) +
    \int_0^T (\bar{e},g) \, dt =
    - \int_0^T (R(\pi u,\cdot),\Phi) \, dt,
  \end{displaymath}
  where $\bar{e} \equiv U - \pi u$.
\end{lemma}

\begin{lemma}[error representation for \mdgq]
  \label{lem:errorrepresentation,dg}
  Let $U$ be the \mdgq\ solution of \emph{(\ref{eq:u'=f})},
  let $\Phi$ be the corresponding \mdgqd\ solution of the dual problem \emph{(\ref{eq:dual})}, and
  let $\pi u$ be any trial space approximation of the exact solution $u$ of \emph{(\ref{eq:u'=f})} that
  interpolates $u$ at the right end-point of every local interval.
  Then
  \begin{displaymath}
    L_{\psi,g}(\bar{e})
    = - \sum_{i=1}^N \sum_{j=1}^{M_i}
    \left[
      [\pi u_i]_{i,j-1} \Phi_i\big(t_{i,j-1}^+\big) +
      \int_{I_{ij}} R_i(\pi u,\cdot) \Phi_i \, dt
      \right],
  \end{displaymath}
  where $\bar{e} \equiv U - \pi u$.
\end{lemma}

With a special choice of interpolant, \smash{$\pi u = \picgq u$} and \smash{$\pi u = \pidgq u$,}
respectively, we obtain
the following versions of the error representations.

\begin{corollary}[error representation for \mcgq] \label{cor:errorrepresentation,cg}
  Let $U$ be the \mcgq\ solution of \emph{(\ref{eq:u'=f})} and
  let $\Phi$ be the corresponding \mcgqd\ solution of the dual problem \emph{(\ref{eq:dual})}.
  Then
  \begin{displaymath}
    L_{\psi,g}(\bar{e})
    =
    \int_0^T \big(f\big(\picgq u,\cdot\big) - f(u,\cdot),\Phi\big) \, dt.
  \end{displaymath}
\end{corollary}\unskip

\begin{proof}
  Integrate by parts and use the definition of the interpolant $\picgq$.\qquad
\end{proof}

\begin{corollary}[error representation for \mdgq] \label{cor:errorrepresentation,dg}
  Let $U$ be the \mdgq\ solution of \emph{(\ref{eq:u'=f})} and
  let $\Phi$ be the corresponding \mdgqd\ solution of the dual problem \emph{(\ref{eq:dual})}.
  Then
  \begin{displaymath}
    L_{\psi,g}(\bar{e})
    =
    \int_0^T \big(f\big(\pidgq u,\cdot\big) - f(u,\cdot),\Phi\big) \, dt.
  \end{displaymath}
\end{corollary}\unskip

\begin{proof}
  Integrate by parts and use the definition of the interpolant $\pidgq$.\qquad
\end{proof}

\subsection{A priori error estimates for the general nonlinear problem}
\label{sec:aprioriestimates}

Using the error representations of section
\ref{sec:errorrepresentation}, the stability estimates of section
\ref{sec:stability}, and the interpolation estimates of section
\ref{sec:interpolation}, we now prove our main results: a priori error
estimates for general order \mcgq{} and \mdgq{}.

\begin{theorem}[a priori error estimate for \mcgq] \label{th:apriori,cg}
  Let $U$ be the \mcgq\ solution of \emph{(\ref{eq:u'=f})} and
  let $\Phi$ be the corresponding \mcgqd\ solution of the dual problem \emph{(\ref{eq:dual})}.
  Then there is a constant $C=C(q)>0$ such that
  \begin{equation} \label{eq:estimate,cg,1}
    \left|L_{\psi,g}(\bar{e})\right|
    \leq
    C S(T) \big\|k^{q+1} \bar{u}^{(q+1)}\big\|_{L_{\infty}([0,T],l_2)},
  \end{equation}
  where $(k^{q+1}\bar{u}^{(q+1)})_i(t) = k_{ij}^{q_{ij}+1} \|u_i^{(q_{ij}+1)}\|_{L_{\infty}(I_{ij})}$
  for $t\in I_{ij}$, and
  where the stability factor $S(T)$ is given by $S(T) = \int_0^T \|J^{\top}(\picgq u, u, \cdot) \Phi\|_{l_2} \, dt$.
  Furthermore, if assumptions \emph{(A1)}--\emph{(A5)} hold, then there is a constant $C=C(q,c_k,\alpha)>0$ such that
  \begin{equation} \label{eq:estimate,cg,2}
    \left| L_{\psi,g}(\bar{e}) \right|
    \leq
    C \bar{S}(T) \big\|k^{2q} \bar{\bar{u}}^{(2q)}\big\|_{L_{\infty}([0,T],l_1)},
  \end{equation}
  where $(k^{2q}\bar{\bar{u}}^{(2q)})_i(t) = k_{ij}^{2q_{ij}} C_f^{q_{ij}-1} \|u_i^{(q_{ij}+1)}\|_{L_{\infty}(I_{ij})}$
  for $t\in I_{ij}$,
  and where the stability factor $\bar{S}(T)$ is given by\vspace*{-3pt}
  \begin{displaymath}
    \bar{S}(T)
    = \int_0^T C_f \|\Phi\|_{L_{\infty}(\mathcal{T},l_{\infty})} \, dt
    = \sum_{n=1}^M K_n C_f \|\Phi\|_{L_{\infty}(\mathcal{T}_n,l_{\infty})}.\vspace*{-3pt}
  \end{displaymath}
\end{theorem}\unskip

\begin{proof}
  By Corollary \ref{cor:errorrepresentation,cg}, we obtain
  \begin{displaymath}
    L_{\psi,g}(\bar{e})
    = \int_0^T \big(f\big(\picgq u,\cdot\big) - f(u,\cdot),\Phi\big) \, dt
    = \int_0^T \big(\picgq u - u, J^{\top}\big(\picgq u,u,\cdot\big) \Phi\big) \, dt.
  \end{displaymath}
  By Lemma \ref{lem:estimate,piecewise}, it now follows that
  \begin{displaymath}
    \left|L_{\psi,g}(\bar{e})\right|
    \leq
    C \|k^{q+1} \bar{u}^{q+1}\|_{L_{\infty}([0,T],l_2)} \int_0^T \big\|J^{\top}\big(\picgq u, u, \cdot\big) \Phi\big\|_{l_2} \, dt,
  \end{displaymath}
  which proves (\ref{eq:estimate,cg,1}). To prove (\ref{eq:estimate,cg,2}), we note that
  by definition, $\picg^{[q_{ij}]} u_i - u_i$ is orthogonal
  to $\mathcal{P}^{q_{ij}-2}(I_{ij})$ for each local interval $I_{ij}$, and so,
  recalling that $\varphi = J^{\top}(\picgq u, u, \cdot) \Phi$,
  \begin{eqnarray*}
    L_{\psi,g}(\bar{e})
    &=&
    \sum_{i,j}
    \int_{I_{ij}} \big(\picg^{[q_{ij}]} u_i - u_i\big) \varphi_i \, dt
    =
    \sum_{i,j}
    \int_{I_{ij}} \big(\picg^{[q_{ij}]} u_i - u_i\big) \big(\varphi_i - \picg^{[q_{ij}-2]} \varphi_i\big) \, dt,
  \end{eqnarray*}
  where we take $\picg^{[q_{ij}-2]} \varphi_i \equiv 0$ for $q_{ij} = 1$.
  By Lemmas \ref{lem:estimate,piecewise} and \ref{lem:interpolationestimate,phi}, it
  now follows that
  \begin{displaymath}
    \begin{split}
      \left| L_{\psi,g}(\bar{e}) \right|
      &\leq
      \int_0^T \big|\big(\picg^{[q]} u - u,\varphi - \picg^{[q-2]} \varphi\big)\big| \, dt \\
      &=
      \int_0^T \big|\big(k^{q-1} C_f^{q-1} \big(\picg^{[q]} u - u\big), k^{-(q-1)} C_f^{-(q-1)} \big(\varphi - \picg^{[q-2]}
\varphi\big)\big)\big| \, dt \\
      &\leq
      C \big\|k^{2q} \bar{\bar{u}}^{(2q)}\big\|_{L_{\infty}([0,T],l_1)}
      \int_0^T C_f \|\Phi\|_{L_{\infty}(\mathcal{T},l_{\infty})} \, dt \\
      &= C \bar{S}(T) \big\|k^{2q} \bar{\bar{u}}^{(2q)}\big\|_{L_{\infty}([0,T],l_1)},
    \end{split}
  \end{displaymath}
  where
  $\bar{S}(T)
    = \int_0^T C_f \|\Phi\|_{L_{\infty}(\mathcal{T},l_{\infty})} \, dt
    = \sum_{n=1}^M K_n C_f \|\Phi\|_{L_{\infty}(\mathcal{T}_n,l_{\infty})}$.\qquad
\end{proof}

Similarly, we obtain the following a priori error estimate for the
\mdgq{} method.

\begin{theorem}[a priori error estimate for \mdgq] \label{th:apriori,dg}
  Let $U$ be the \mdgq\ solution of \emph{(\ref{eq:u'=f})} and
  let $\Phi$ be the corresponding \mdgqd\ solution of the dual problem \emph{(\ref{eq:dual})}.
  Then there is a constant $C=C(q)>0$ such that
  \begin{equation} \label{eq:estimate,dg,1}
    \left|L_{\psi,g}(\bar{e})\right|
    \leq
    C S(T) \big\|k^{q+1} \bar{u}^{(q+1)}\big\|_{L_{\infty}([0,T],l_2)},
  \end{equation}
  where $(k^{q+1}\bar{u}^{(q+1)})_i(t) = k_{ij}^{q_{ij}+1} \|u_i^{(q_{ij}+1)}\|_{L_{\infty}(I_{ij})}$
  for $t\in I_{ij}$, and
  where the stability factor $S(T)$ is given by $S(T) = \int_0^T \|J^{\top}(\pidgq u, u, \cdot) \Phi\|_{l_2} \, dt$.
  Furthermore, if assumptions \emph{(A1)}--\emph{(A5)} hold, then there is a constant $C=C(q,c_k,\alpha)>0$ such that
  \begin{equation} \label{eq:estimate,dg,2}
    \left| L_{\psi,g}(\bar{e}) \right|
    \leq
    C \bar{S}(T) \big\|k^{2q+1} \bar{\bar{u}}^{(2q+1)}\big\|_{L_{\infty}([0,T],l_1)},
  \end{equation}
  where $(k^{2q+1}\bar{\bar{u}}^{(2q+1)})_i(t) = k_{ij}^{2q_{ij}+1} C_f^{q_{ij}} \|u_i^{(q_{ij}+1)}\|_{L_{\infty}(I_{ij})}$
  for $t\in I_{ij}$,
  and where the stability factor $\bar{S}(T)$ is given by
  \begin{displaymath}
    \bar{S}(T)
    = \int_0^T C_f \|\Phi\|_{L_{\infty}(\mathcal{T},l_{\infty})} \, dt
    = \sum_{n=1}^M K_n C_f \|\Phi\|_{L_{\infty}(\mathcal{T}_n,l_{\infty})}.
  \end{displaymath}
\end{theorem}\unskip

Using the stability estimate proved in section \ref{sec:stability}, we obtain the following bound for the stability
factor $\bar{S}(T)$.

\begin{lemma} \label{lem:stability}
  Assume that $K_n C_q C_f \leq 1$ for all time slabs $\mathcal{T}_n$,
  with $C_q > 0$ the constant in Theorem {\rm \ref{th:estimate,exponential}}, and take $g=0$ in {\rm (\ref{eq:dual})}.
  Then
  \begin{equation}
    \bar{S}(T) \leq \|\psi\|_{l_{\infty}} e^{ C_q \bar{C}_f T },
  \end{equation}
  where $\bar{C}_f = \max_{[0,T]} C_f$.
\end{lemma}

{\it Proof.}
  By Theorem \ref{th:estimate,exponential}, we obtain
  \begin{displaymath}
    \|\Phi\|_{L_{\infty}(\mathcal{T}_n,l_{\infty})} \leq C_q
    \|\psi\|_{l_{\infty}} \exp\left( \sum_{m=n+1}^{M} K_m C_q C_f \right)
    \leq
    C_q \|\psi\|_{l_{\infty}} e^{ C_q \bar{C}_f (T-T_n) },
  \end{displaymath}
  and so
  \begin{displaymath}
    \begin{split}
      \bar{S}(T)
      &= \sum_{n=1}^M K_n C_f \|\Phi\|_{L_{\infty}(\mathcal{T}_n,l_{\infty})} \, dt
      \leq \|\psi\|_{l_{\infty}} \sum_{n=1}^M K_n C_q \bar{C}_f e^{ C_q \bar{C}_f (T-T_n) } \\
      &\leq
      \|\psi\|_{l_{\infty}} \int_0^T C_q \bar{C}_f e^{ C_q \bar{C}_f t } \, dt
      \leq
      \|\psi\|_{l_{\infty}} e^{ C_q \bar{C}_f T }.\qquad\endproof
    \end{split}
  \end{displaymath}

Finally, we rewrite the estimates of Theorems \ref{th:apriori,cg} and \ref{th:apriori,dg} for special
choices of data $\psi$ and $g$.
We first take $\psi=0$. With $g_n=0$ for $n \neq i$, $g_i(t)=0$ for $t\not\in I_{ij}$, and
\begin{displaymath}
  g_i(t) = \sgn(\bar{e}_i(t))/k_{ij}, \quad t\in I_{ij},
\end{displaymath}
we obtain \smash{$L_{\psi,g}(\bar{e}) = \frac{1}{k_{ij}} \int_{I_{ij}} |\bar{e}_i(t)| \, dt$} and
so {$\|\bar{e}_i\|_{L_{\infty}(I_{ij})} \leq C L_{\psi,g}(\bar{e})$} by an inverse estimate.
By definition, it follows that
{$\|e_i\|_{L_{\infty}(I_{ij})} \leq C L_{\psi,g}(\bar{e}) + C k_{ij}^{q_{ij}+1} \|u_i^{q_{ij}+1}\|_{L_{\infty}(I_{ij})}$}.
Note that for this choice of $g$, we have {$\|g\|_{L_1([0,T],l_2)}=\|g\|_{L_1([0,T],l_{\infty})}=1$}.

We also make the choice $g=0$. Noting that $\bar{e}(T) = e(T)$, since $\pi u(T) = u(T)$, we obtain
\begin{displaymath}
  L_{\psi,g}(\bar{e}) = (e(T),\psi) = |e_i(T)|
\end{displaymath}
for $\psi_i = \sgn(e_i(T))$ and $\psi_n = 0$ for $n\neq i$,
and
\begin{displaymath}
  L_{\psi,g}(\bar{e}) = (e(T),\psi) = \|e(T)\|_{l_2}
\end{displaymath}
for $\psi = e(T)/\|e(T)\|_{l_2}$. Note that for both choices of $\psi$, we have $\|\psi\|_{l_{\infty}}\leq 1$.

With these choices of data, we obtain the following versions of the a priori error estimates.

\begin{corollary}[a priori error estimate for \mcgq] \label{cor:apriori,cg}
  Let $U$ be the \mcgq\ solution of \emph{(\ref{eq:u'=f})}.
  Then there is a constant $C=C(q)>0$ such that
  \begin{equation}
    \|e\|_{L_{\infty}([0,T],l_{\infty})}
    \leq
    C S(T) \big\|k^{q+1} \bar{u}^{(q+1)}\big\|_{L_{\infty}([0,T],l_2)},
  \end{equation}
  where the stability factor $S(T) = \int_0^T \|J^{\top}(\picgq u, u, \cdot) \Phi\|_{l_2} \, dt$ is taken as the maximum
  over $\psi=0$ and $\|g\|_{L_1([0,T],l_{\infty})}=1$.
  Furthermore, if assumptions \emph{(A1)}--\emph{(A5)} and the assumptions of Lemma~{\rm \ref{lem:stability}} hold,
  then there is a constant $C=C(q,c_k,\alpha)$ such that
  \begin{equation}
    \|e(T)\|_{l_p}
    \leq
    C \bar{S}(T) \big\|k^{2q} \bar{\bar{u}}^{(2q)}\big\|_{L_{\infty}([0,T],l_1)}
  \end{equation}
  for $p=2,\infty$, where the stability factor $\bar{S}(T)$ is given by $\bar{S}(T) = e^{C_q \bar{C}_f T}$.
\end{corollary}

\begin{corollary}[a priori error estimate for \mdgq] \label{cor:apriori,dg}
  Let $U$ be the \mdgq\ solution of \emph{(\ref{eq:u'=f})}.
  Then there is a constant $C=C(q)>0$ such that
  \begin{equation}
    \|e\|_{L_{\infty}([0,T],l_{\infty})}
    \leq
    C S(T) \big\|k^{q+1} \bar{u}^{(q+1)}\big\|_{L_{\infty}([0,T],l_2)},
  \end{equation}
  where the stability factor $S(T) = \int_0^T \|J^{\top}(\pidgq u, u, \cdot) \Phi\|_{l_2} \, dt$ is taken as the maximum
  over $\psi=0$ and $\|g\|_{L_1([0,T],l_{\infty})}=1$.
  Furthermore, if assumptions \emph{(A1)}--\emph{(A5)} and the assumptions of Lemma~{\rm \ref{lem:stability}} hold,
  then there is a constant $C=C(q,c_k,\alpha)$ such that
  \begin{equation}
    \|e(T)\|_{l_p}
    \leq
    C \bar{S}(T) \big\|k^{2q+1} \bar{\bar{u}}^{(2q+1)}\big\|_{L_{\infty}([0,T],l_1)}
  \end{equation}
  for $p=2,\infty$, where the stability factor $\bar{S}(T)$ is given by $\bar{S}(T) = e^{C_q \bar{C}_f T}$.
\end{corollary}

The stability factor $S(T)$ that appears in the a priori error
estimates is obtained from the discrete solution $\Phi$ of the dual
problem (\ref{eq:linear,dual}), and can thus be computed by
solving the discrete dual problem. Numerical computation of the
stability factor reveals the exact nature of the problem, in
particular, whether or not the problem is parabolic; if the stability
factor is of unit size and does not grow, then the problem is
parabolic by definition; see \cite{logg:article:05}.

\subsection{A note on quadrature errors}

The error representations presented in section
\ref{sec:errorrepresentation} are based on the
Galerkin orthogonalities of the \mcgq\ and \mdgq\ methods.
In particular, for the \mcgq{} method, we assume that
\begin{displaymath}
  \int_0^T (R(U,\cdot), \Phi) \, dt = 0.
\end{displaymath}
In the presence of quadrature errors, this term is nonzero. As a
result, we obtain an additional term of the form
\begin{displaymath}
  \int_0^T (\tilde{f}(U,\cdot) - f(U,\cdot),\Phi) \, dt,
\end{displaymath}
where $\tilde{f}$ is the interpolant of $f$ corresponding the
quadrature rule that is used. A convenient choice of quadrature for the
\mcgq{} method is Lobatto quadrature with $q\,{+}\,1$
nodal points \cite{logg:article:01}, which means that the
quadrature error is of order $2(q+1)-2 = 2q$ and so (super)convergence
of order $2q$ is obtained also in the presence of quadrature
errors. Similarly for the \mdgq\ method, we use Radau quadrature with
$q+1$ nodal points, which means that the quadrature error is of
order $2(q+1)-1 = 2q + 1$, and so the $2q+1$ convergence order of \mdgq{}
is also maintained under quadrature.

\section{A numerical example}
\label{sec:numerical}

We conclude by demonstrating the convergence of the multiadaptive methods in the
case of a simple test problem.

Consider the problem
\begin{equation} \label{eq:test}
  \begin{split}
    \dot{u}_1 &= u_2,                 \\
    \dot{u}_2 &= - u_1,               \\
    \dot{u}_3 &= -u_2 + 2 u_4,         \\
    \dot{u}_4 &= u_1 - 2 u_3,          \\
    \dot{u}_5 &= -u_2 - 2 u_4 + 4 u_6, \\
    \dot{u}_6 &= u_1 + 2 u_3 - 4 u_5  \\
  \end{split}
\end{equation}
on $[0,1]$ with initial condition $u(0) = (0,1,0,2,0,3)$. The solution
is given by $u(t) = (\sin t,\cos t,\sin t + \sin 2t, \cos t + \cos 2t,\sin t + \sin 2t + \sin 4t,\cos t + \cos 2t + \cos 4t)$.
For given $k_0>0$, we take $k_i(t) = k_0$ for $i=1,2$, $k_i(t) = k_0/2$ for $i=3,4$, and $k_i(t) = k_0/4$ for $i=5,6$,
and study the convergence of the error $\|e(T)\|_{l_2}$ with decreasing $k_0$. From the results presented
in Figure \ref{fig:convergence} and  Tables \ref{tab:convergence,cg} and \ref{tab:convergence,dg},
it is clear that the predicted order of convergence is
obtained.

\begin{figure}[t]
  \begin{center}
    \psfrag{k}{\small $k_0$}
    \psfrag{e}{\hspace{-0.5cm}\small $\|e(T)\|_{l_2}$}
    \psfrag{cg1}{\small $\mathrm{cG}(1)$}
    \psfrag{cg2}{\small $\mathrm{cG}(2)$}
    \psfrag{cg3}{\small $\mathrm{cG}(3)$}
    \psfrag{cg4}{\small $\mathrm{cG}(4)$}
    \psfrag{cg5}{\small $\mathrm{cG}(5)$}
    \psfrag{dg0}{\small $\mathrm{dG}(0)$}
    \psfrag{dg1}{\small $\mathrm{dG}(1)$}
    \psfrag{dg2}{\small $\mathrm{dG}(2)$}
    \psfrag{dg3}{\small $\mathrm{dG}(3)$}
    \psfrag{dg4}{\small $\mathrm{dG}(4)$}
    \psfrag{dg5}{\small $\mathrm{dG}(5)$}
    \includegraphics[width=10cm]{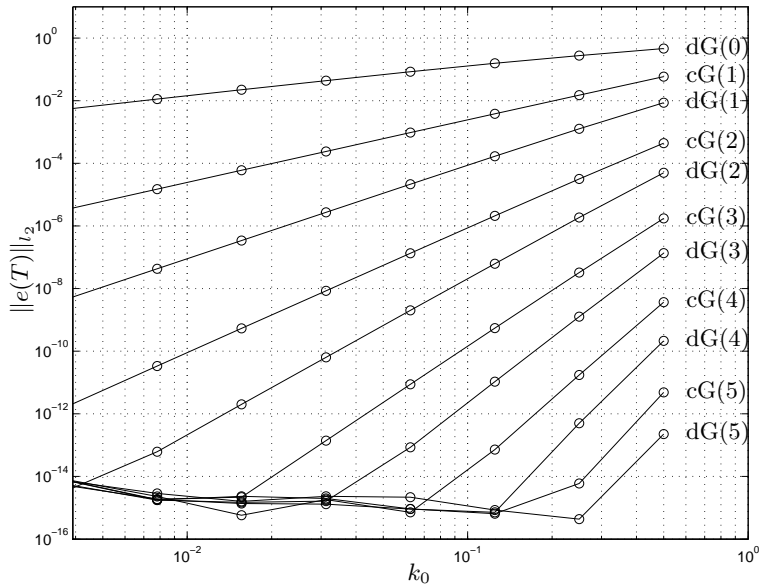}
    \caption{Convergence of the error at final time for the solution of the test problem
    {\rm (\ref{eq:test})} with \mcgq\ and \mdgq, $q\leq 5$.}
    \label{fig:convergence}
  \end{center}
\end{figure}

\begin{table}[t!]
\footnotesize{
\caption{Order of convergence $p$ for \mcgq.}
  \label{tab:convergence,cg}
\begin{center}
    \begin{tabular}{|c||c|c|c|c|c|}
      \hline
      \mcgq & $1$    & $2$    & $3$    & $4$    & $5$    \\
      \hline
      $p$   & $1.99$ & $3.96$ & $5.92$ & $7.82$ & $9.67$ \\
      \hline
      $2q$  & $2$    & $4$    & $6$    & $8$    & $10$   \\
      \hline
    \end{tabular}
  \end{center}}
  \end{table}

\begin{table}[t]
\footnotesize{
  \caption{Order of convergence $p$ for \mdgq.}
  \label{tab:convergence,dg}
  \begin{center}
    \begin{tabular}{|c||c|c|c|c|c|c|}
      \hline
      \mdgq  & $0$    & $1$    & $2$    & $3$    & $4$    & $5$ \\
      \hline
      $p$    & $0.92$ & $2.96$ & $4.94$ & $6.87$ & $9.10$ & --    \\
      \hline
      $2q+1$ & $1$    & $3$    & $5$    & $7$    & $9$    & $11$   \\
      \hline
    \end{tabular}
  \end{center}}
\end{table}


\bibliographystyle{siam}


\end{document}